\RequirePackage{fix-cm}
\documentclass[twocolumn]{svjour3}          
\smartqed  

\usepackage{amsmath,amsfonts,amssymb}
\usepackage{graphicx}
\usepackage{wrapfig}
\usepackage{url}
\usepackage{multirow}
\usepackage[titletoc]{appendix}
\usepackage{mdframed}
\usepackage{fancyvrb}
\usepackage{tabu}
\usepackage{tabularx}
\usepackage{booktabs}
\usepackage{multirow}
\usepackage{xspace}
\usepackage{enumitem}
\usepackage{algorithm}
\usepackage{algorithmic}
\usepackage{subcaption}
\usepackage{hyperref}
\usepackage{tikz}
\usepackage{ctable}

\usepackage{comment}

\usepackage{dsfont}

\newcommand{\sign}{\text{sign}}

\usepackage{blindtext}
\usepackage[T1]{fontenc}
\usepackage[utf8]{inputenc}

\DeclareMathOperator*{\argmin}{arg\,min}

\journalname{International Journal of Computing and Visualization in Science and Engineering}

\newcommand{\MG}[1]{\textcolor{magenta}{#1}}

\begin{document}

\title{An Optimized Space-Time Multigrid Algorithm for Parabolic PDEs}

\author{Bastien Chaudet-Dumas \and Martin J. Gander \and Au\v{s}ra Pogo\v{z}elskyt\.{e}}

\authorrunning{B. Chaudet-Dumas, M. J. Gander, A. Pogo\v{z}elskyt\.{e}} 


\institute{Bastien Chaudet-Dumas \at
           Section of Mathematics, University of Geneva\\
           \email{bastien.chaudet@unige.ch}
           \and
           Martin J. Gander \at
           Section of Mathematics, University of Geneva\\
           \email{martin.gander@unige.ch}
           \and
           Au\v{s}ra Pogo\v{z}elskyt\.{e} \at
           Section of Mathematics, University of Geneva\\
           \email{ausra.pogozelskyte@unige.ch}
}

\date{}

\maketitle

\begin{abstract}
   We investigate three directions to further improve the highly efficient Space-Time Multigrid algorithm with block-Jacobi smoother introduced in \cite{gander2016analysis}. First, we derive an analytical expression for the optimal smoothing parameter in the case of a full space-time coarsening strategy;
   second, we propose a new and efficient direct coarsening strategy which simplifies the code by preventing changes of coarsening regimes; and third, we also optimize the entire two cycle to investigate if further efficiency gains are possible.
   Especially, we show that our new coarsening strategy leads to a significant efficiency gains when the ratio $\tau/h^2$ is small, where $\tau$ and $h$ represent the time and space steps.
   Our analysis is performed for the heat equation in one spatial dimension, using centered finite differences in space and Backward Euler in time, but could be generalized to other situations. We also present numerical experiments that confirm our theoretical findings.

\end{abstract}

\section{Introduction}

    During the last twenty years, faster computing times have been achieved through the increase of the number of processing cores available rather than the development of faster clock speeds. 
    However, as spatial parallelism techniques saturate, developing parallel techniques for the time dimension becomes important.     
    Popular examples of such methods include Parareal~\cite{lions2001resolution} and its multilevel version: Multigrid Reduction-in-Time (MGRIT) \cite{falgout2014parallel}.
    For a review on Parallel-in-time methods, the reader is referred to~\cite{gander201550,ong2020applications,falgout2017multigrid}.
    Yet, many of the methods studied in these references coarsen the problem only in the time dimension. Therefore they are not truly scalable when applied to space-time problems. In order to overcome this issue, new methods that also include spatial coarsening have been developed, see e.g.~\cite{horton1995space,fischer2005parareal,neumuller2013space,gander2016analysis,ruprecht2014convergence,schroder2021,gander2022}.

    Among these space-time approaches, we are interested in a specific, efficient Space-Time Multigrid algorithm (STMG), originally introduced
    in~\cite{gander2016analysis} based on a block-Jacobi smoother. This method, which benefits from excellent scalability properties, has been successfully applied to several problems~\cite{neumuller2013space,neumuller2018fully,neumuller2019parallel}. 
    
In~\cite{gander2016analysis}, one of the goals of the authors was
  to develop a space-time multigrid method based only on standard
  components that always permits coarsening in the time direction, in
  order to overcome a limitation in the early parabolic multigrid
  method \cite{Hackbusch:1984:PMG}, leading to a faster and more
  scalable multigrid algorithm. In doing so, in~\cite{gander2016analysis}, emphasis was put on semi-coarsening in time during the optimization of the smoother with respect to the damping parameter. The analysis then reveals that the ratio
$\sigma:=\frac{\tau}{h^2}$, where $\tau$ and $h$ denote the time and
space steps, plays a key role in the convergence of the
method. Especially, in order to guarantee fast convergence, specific
coarsening strategies are needed for this ratio to stay above a given threshold. The authors in \cite{gander2016analysis} then propose a strategy based on alternating semi-coarsening in
time, and full space-time coarsening.

    Our goal here is to further optimize the STMG algorithm from~\cite{gander2016analysis} in three ways. First, we optimize the damping parameter of the block-Jacobi smoother for the case of full space-time coarsening. Second, we propose a new coarsening strategy to control the ratio $\sigma$, which results in a simpler algorithm while maintaining efficiency. And third, we investigate if optimizing the entire two-level process like in \cite{lucero2020optimization} can lead to further gains.
    The whole analysis is conducted in the specific case of the heat equation in one spatial dimension discretized using centered finite differences in space and the Backward Euler method in time, but our approach could be generalized to higher dimensions and other equations.
    
    We begin with introducing the continuous and discrete versions of the model problem in Section~\ref{sec:model-prob}. Then, in Section~\ref{sec:STMG}, we present the STMG algorithm in our specific context. In Section~\ref{sec:analysis-tech}, we review some useful notations and results from Local Fourier Analysis. Next, we analyze the performance of the smoother and derive the optimal damping parameter for full space-time coarsening in Section~\ref{sec:smoother}. Section~\ref{sec:coarse-grid-analysis} is dedicated to the comparison of different coarsening strategies preventing $\sigma$ from becoming too small, as well as to the optimization of the whole two-level process. Finally, numerical results are presented in Section~\ref{sec:num-exp}.

    \section{Model problem}
    \label{sec:model-prob}

    We consider as our parabolic model problem the one-dimensional heat equation. Let $T>0$, the space-time domain is given by $\Omega\times(0,T)$ where $\Omega:=(0,1)$. Given a source term $f$ and an initial condition $u_0$, the strong formulation of the problem reads
    \begin{equation}
        \label{eq:continuous-heat-equation}
        \begin{cases}
            \partial_{t}u(x, t) = \partial_{xx}u(x, t) + f(x, t), & (x, t)\in\Omega\times(0, T), \\ 
            u(x, t) = 0, & x\in\partial\Omega, \\ 
            u(x, 0) = u_{0}(x), & x\in \Omega.
        \end{cases}
    \end{equation}
    By \cite[Chapter 4]{lions2012non}, we know that if $f \in L^2(0,T;H^{-1}(\Omega))$ and $u_0\in L^2(\Omega)$, then there exists a unique solution $u\in L^2(0,T;H^1_0(\Omega))\cap C^0(0,T;L^2(\Omega))$ to problem \eqref{eq:continuous-heat-equation}.
    Under such regularity assumptions on the data, we may discretize the problem in space using a step $h$, which leads to the matrix differential equation
        \begin{equation}
        \label{eq:heat-discretized-in-space}
        \begin{cases}
            \partial_{t}\boldsymbol{u}(t) = A_{h}\boldsymbol{u}(t) + \boldsymbol{f}(t), & t\in (0, T),\\ 
            \boldsymbol{u}(0) = \boldsymbol{u}_{0},
        \end{cases}
    \end{equation}
    where $A_h\in \mathbb{R}^{N_x\times N_x}$ is the discrete Laplace operator and $N_x$ is the number of space steps.
    Problem~\eqref{eq:heat-discretized-in-space} can then be discretized in time with a time step $\tau$ using a Runge-Kutta method. Let $R(z)=: Q^{-1}(z)B(z)$ be the associated stability function~\cite{wanner1996solving}, where $Q$ and $B$ are given polynomials. 
    For simplicity, we introduce $Q_{\tau, h} := Q(\tau A_h)$ and $B_{\tau, h}:=B(\tau A_h)$. 
    Using these notations, we finally obtain
    \begin{equation}
        \label{eq:discrete-time-stepping}
        \begin{cases}
            Q_{\tau, h}\boldsymbol{u}_{n+1} = B_{\tau, h}\boldsymbol{u}_{n} + \boldsymbol{f}_{n+1} & n= 0, \ldots, N_t-1,\\ 
            \boldsymbol{u}_{0} = \boldsymbol{u}_{0},
        \end{cases}
    \end{equation}
    where $N_t$ denotes the number of time steps.
    The linear system~\eqref{eq:discrete-time-stepping} can then be rewritten as a matrix system,
    \begin{equation}\label{eq:discrete-system}
        \underbrace{
        \begin{pmatrix}
            Q_{\tau, h} \\ 
            -B_{\tau, h} & Q_{\tau, h} \\ 
            & \ddots & \ddots \\ 
            && -B_{\tau, h} & Q_{\tau, h}
        \end{pmatrix}}_{=:L_{\tau,h}}
        \underbrace{
        \begin{pmatrix}
            \boldsymbol{u}_{1} \\ 
            \boldsymbol{u}_{2} \\
            \vdots\\
            \boldsymbol{u}_{N_t}
        \end{pmatrix}}_{=:\boldsymbol{u}}
        = \underbrace{\begin{pmatrix}
            \boldsymbol{f}_{1}+ B_{\tau, h}\boldsymbol{u}_{0} \\ 
            \boldsymbol{f}_{2}\\ 
            \vdots \\
            \boldsymbol{f}_{N_t}
        \end{pmatrix}}_{=: \boldsymbol{f}}\;.
    \end{equation}

\section{Space-Time Multigrid}
\label{sec:STMG}

    System~\eqref{eq:discrete-system} can be solved using the STMG algorithm. It follows the same steps as traditional multigrid algorithms~\cite{trottenberg2000multigrid}, which yields in the case of one simple V-cycle:  
    \begin{enumerate}
        \item pre-smoothing ($\nu_1$ steps); 
        \item computation of the residual and restriction to the coarse grid;
        \item coarse grid solve;
        \item prolongation to the fine grid and correction;
        \item post-smoothing ($\nu_2$ steps).
    \end{enumerate}
    In the following subsections, we will describe each of the different components of the algorithm in detail. 
    
    \subsection{Smoothing}
    Each iteration of STMG starts with $\nu_{1}$ pre-smoothing steps and ends with $\nu_{2}$ post-smoothing steps. Smoothing is used to damp highly oscillating modes in the error so that it can be properly represented on the coarse grid. Following \cite{gander2016analysis}, we use a damped block-Jacobi smoother, 
    \begin{equation}
        \label{eq:def-smoother}
        S_{\tau,h}^{\nu} = [I_{\tau,h} - \omega D_{\tau,h}^{-1}L_{\tau,h}]^{\nu}\;,
    \end{equation}
    where $\omega$ is the damping parameter and $D_{\tau,h}$ is a block-diagonal square matrix of size $N_tN_x$ composed of $N_t$ times the block $Q_{\tau,h}$ on the diagonal. 

    \begin{remark}
        Let $\Gamma_{-1,\tau}$ be the $N_t\times N_t$ matrix whose entries are ones on the lower-diagonal and zeros elsewhere. Using the definition of $L_{\tau, h}$ in~\eqref{eq:discrete-system}, the block-Jacobi smoother is given by 
        \begin{equation*}
            S_{\tau, h} = [I_{\tau}\otimes (1 - \omega)I_{h} - \omega\Gamma_{-1,\tau}\otimes Q_{\tau, h}^{-1}B_{\tau, h} ]\;.
        \end{equation*}
    \end{remark}

    \begin{remark}
      For the smoothing iterative procedure to be convergent, we need the spectral radius of the smoother $S_{\tau,h}$ to verify $\rho(S_{\tau, h}) = |1 - \omega|<1$, that is $\omega\in(0, 2)$. 
      In the following, we will restrict ourselves to case $\omega\in(0, 1]$ in order to avoid non-uniform convergence, see~\cite{gander2016analysis}.
    \end{remark}

    \subsection{Transfer operators}
    To go from the coarse to the fine grid and back, we need
    restriction and prolongation operators for both space and time
    grids, which we generically denote by $R_h$, $P_h$ and
    $R_{\tau}$, $P_{\tau}$. Their definitions depend on the type
    of coarsening we do, and are sometimes identities:
    \begin{enumerate}
        \item semi-coarsening in space $x$: $P_{\tau} = R_{\tau} = I_{\tau}$;
        \item semi-coarsening in time $t$: $P_{h} = R_{h} = I_{h}$; 
        \item full-coarsening in space and time;
        \item new coarsening strategy (see details below).        
    \end{enumerate}
    Unless specified otherwise, when the transfer operators are not equal to the identity, the prolongation operator is taken as the linear interpolation and the restriction operator is computed using the full-weighting technique.
    Note that using different transfer operators would result in a similar analysis.

    The approach used in~\cite{gander2016analysis} is to coarsen by the same factor in space and time at each level. At some point, the smoothing operator does not smooth very well (this will be further discussed in Section~\ref{sec:smoother}) and we switch to doing semi-coarsening in time.
    To simplify the analysis, this technique will consist of a factor $2$ space-time coarsening followed by a semi-coarsening in time (see the ''Original strategy'' in Figure~\ref{fig:coarsening-strategies}). This is a slower technique than the one initially proposed but it gives us a basis for comparison. 

    The coarsening technique we propose consists in coarsening by a factor $4$ in time and $2$ in space, therefore maintaining the ratio $\sigma = \tau/h^2$ constant (see the ''New strategy'' in Figure~\ref{fig:coarsening-strategies}).
    This will yield a computationally cheaper strategy, while maintaining the convergence properties of the original method, therefore making the STMG algorithm even more efficient.

    \subsection{Coarse-grid solve}
    
    The restricted residual is solved on the coarser grid, meaning that we either perform a direct solve (two-level method) or we repeat steps 1 to 5 on this coarser grid until we reach a grid for which a direct solve is considered to be cheap (multi-level method). 

    For the coarse problem, we have two options: either use a Galerkin approximation of the operator $L_{\tau, h}$, that is e.g. $L_{2\tau, 2h} = (R_{\tau}\otimes R_{h})L_{\tau, h}(P_{\tau}\otimes P_{h})$,  or a rediscretization with coarser space and/or time steps, in which case $L_{2\tau, 2h} = I_{2\tau}\otimes Q_{2\tau,2h}+\Gamma_{-1,2\tau}\otimes B_{2\tau,2h}$.
    In our analysis, we will consider rediscretization.

    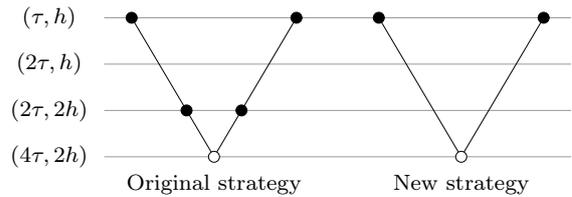
\begin{figure}
        \centering
        \resizebox{0.9\columnwidth}{!}{
            \begin{tikzpicture}
                \def\h{1.7}; 
                \def\k{1}; 
                \def\txt{2.7}; 

                \node[scale=\txt] at (-2, 3*\h) {$(\tau, h)$};
                \node[scale=\txt] at (-2, 2*\h) {$(2\tau, h)$};
                \node[scale=\txt] at (-2, 1*\h) {$(2\tau, 2h)$};
                \node[scale=\txt] at (-2, 0*\h) {$(4\tau, 2h)$};

                \draw[black!40] (0, 0*\h) -- (17*\k, 0*\h); 
                \draw[black!40] (0, 1*\h) -- (17*\k, 1*\h); 
                \draw[black!40] (0, 2*\h) -- (17*\k, 2*\h); 
                \draw[black!40] (0, 3*\h) -- (17*\k, 3*\h); 

                \draw (1*\k , 3*\h) -- (4*\k , 0*\h) -- (7*\k , 3*\h); 
                \draw[-latex] (10*\k, 3*\h) -- (13*\k, 0*\h); 
                \draw[-latex] (13*\k, 0*\h) -- (16*\k, 3*\h); 

                \draw[fill=black] (1*\k, 3*\h) circle (0.2cm); 
                \draw[fill=black] (3*\k, 1*\h) circle (0.2cm);
                \draw[fill=white] (4*\k, 0*\h) circle (0.2cm); 
                \draw[fill=black] (5*\k, 1*\h) circle (0.2cm);
                \draw[fill=black] (7*\k, 3*\h) circle (0.2cm); 
                \draw[fill=black] (10*\k, 3*\h) circle (0.2cm); 
                \draw[fill=white] (13*\k, 0*\h) circle (0.2cm); 
                \draw[fill=black] (16*\k, 3*\h) circle (0.2cm); 

                \node[scale=\txt] at (4*\k , -1) {Original strategy}; 
                \node[scale=\txt] at (13*\k, -1) {New strategy};
            \end{tikzpicture}
        }

        \caption{Different coarsening strategies considered in the coarse-grid analysis in Section~\ref{sec:coarse-grid-analysis}.}
        \label{fig:coarsening-strategies}
    \end{figure}

\section{Analysis techniques}
\label{sec:analysis-tech}

    The analysis performed in this article relies on Local Fourier Analysis (LFA), first introduced in~\cite{brandt1977multi} and made rigorous in~\cite{brandt1994rigorous}. 
    For an introduction, we refer the reader to the well written books~\cite[Chapter~3]{trottenberg2000multigrid} and~\cite[Section~8.2]{hackbusch1985multi}.
    To this end, we will assume that our problem is periodic in time and space. 

    First, let us recall the Fourier mode decomposition for the vector $\boldsymbol{u}\in \mathbb{R}^{N_t N_x}$. All the following definitions   
    can be found in \cite{gander2016analysis}.
    \begin{definition}
        Let $\boldsymbol{\Phi}: (-\pi,\pi]^2 \to \mathbb{R}^{N_t N_x}$ be defined by $\Phi_{n,j}(\alpha,\beta) := e^{i n\alpha}e^{i j\beta}$, for $1\leq n\leq N_t$ and $1\leq j \leq N_x$. The vector $\boldsymbol{\Phi}(\theta_t,\theta_x)$ is called the Fourier mode with frequency $(\theta_t,\theta_x)\in\Theta_t\times\Theta_x=:\Theta$ where
        \begin{equation}
        \label{eq:def-freq-domains}
            \begin{aligned}
                \Theta_t &:= \left\{ \frac{2k\pi}{N_t} \: | \: k=1-\frac{N_t}{2}, \dots, \frac{N_t}{2} \right\} \subset (-\pi, \pi]\;, \\
                \Theta_x &:= \left\{ \frac{2k\pi}{N_x} \: | \: k=1-\frac{N_x}{2}, \dots, \frac{N_x}{2} \right\} \subset (-\pi, \pi]\;.
            \end{aligned}
        \end{equation}
    \end{definition}
    \begin{definition}
        Let $\boldsymbol{u}=(\boldsymbol{u}_1, \boldsymbol{u}_2,\dots, \boldsymbol{u}_{N_t})\in\mathbb{R}^{N_tN_x}$.
        For each frequency $(\theta_t,\theta_x)\in\Theta$, the Fourier coefficient $\hat{u}$ associated to $(\theta_t,\theta_x)$ is defined by
        \begin{equation*}
            \hat{u}(\theta_t,\theta_x) := \frac{1}{N_t}\frac{1}{N_x}\sum_{n=1}^{N_t}\sum_{j=1}^{N_x}u_{n,j}\, e^{-i n\theta_t}e^{-i j\theta_x}\;.
        \end{equation*}
        Using these notations, the Fourier mode decomposition of $\boldsymbol{u}$ reads
        \begin{equation}
            \label{eq:Fourier-decomp}
            \boldsymbol{u} = \sum_{\theta_t\in\Theta_t} \sum_{\theta_x\in\Theta_x} \hat{u}(\theta_t,\theta_x) \boldsymbol{\Phi}(\theta_t,\theta_x)\;.
        \end{equation}
    \end{definition}

    The space-time frequency domain $\Theta$ can be decomposed as the disjoint union of low frequency and high frequency domains $\Theta=\Theta^{low}\cup\Theta^{high}$. The definitions of these low/high frequency domains depend on the type of coarsening considered, 
    \begin{itemize}
        \item semi-coarsening in time:\\
        $\Theta^{low} = \Theta \cap (-\pi/2, \pi/2]\times (-\pi,\pi]$;
        \item semi-coarsening in space:\\
        $\Theta^{low} = \Theta \cap (-\pi,\pi]\times (-\pi/2, \pi/2]$;
        \item full space-time coarsening:\\
        $\Theta^{low} = \Theta \cap (-\pi/2, \pi/2]\times (-\pi/2,\pi/2]$;
        \item new coarsening strategy:\\
        $\Theta^{low} = \Theta \cap (-\pi/4, \pi/4]\times (-\pi/2,\pi/2]$.
    \end{itemize}
    In each case, we directly deduce $\Theta^{high}=\Theta\setminus\Theta^{low}$, see Figure \ref{fig:freqs}.

    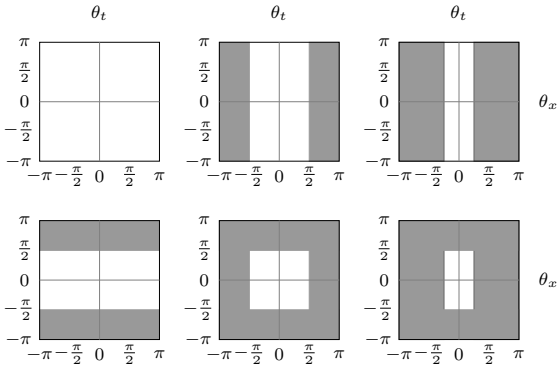
\begin{figure}
        \centering
        \resizebox{0.9\columnwidth}{!}{
            \begin{tikzpicture}
                \def\side{4}; 
                \def\spc{6}; 
                \def\txt{2}; 

                \node[scale=\txt] at (\side/2,       2*\spc-1) {$\theta_t$};
                \node[scale=\txt] at (\spc+\side/2,  2*\spc-1) {$\theta_t$}; 
                \node[scale=\txt] at (2*\spc+\side/2,2*\spc-1) {$\theta_t$}; 
                \node[scale=\txt] at (3*\spc-1, 2*\side) {$\theta_x$}; 
                \node[scale=\txt] at (3*\spc-1, \side/2) {$\theta_x$}; 

                \foreach \j in {0, 1} {

                    \draw[fill=black!40, draw=black!40] (0+1*\spc, 0+\j*\spc) -- (\side/4+1*\spc, 0+\j*\spc) -- (\side/4+1*\spc, \side+\j*\spc) -- (0+1*\spc, \side+\j*\spc) -- cycle; 
                    \draw[fill=black!40, draw=black!40] (3*\side/4+1*\spc, 0+\j*\spc) -- (\side+1*\spc, 0+\j*\spc) -- (\side+1*\spc, \side+\j*\spc) -- (3*\side/4+1*\spc, \side+\j*\spc) -- cycle; 

                    \draw[fill=black!40, draw=black!60] (0+2*\spc, 0+\j*\spc) -- (3*\side/8+2*\spc, 0+\j*\spc) -- (3*\side/8+2*\spc, \side+\j*\spc) -- (0+2*\spc, \side+\j*\spc) -- cycle; 
                    \draw[fill=black!40, draw=black!60] (5*\side/8+2*\spc, 0+\j*\spc) -- (\side+2*\spc, 0+\j*\spc) -- (\side+2*\spc, \side+\j*\spc) -- (5*\side/8+2*\spc, \side+\j*\spc) -- cycle; 
                    
                }

                \foreach \i in {0, 1, 2}{
                    \draw[fill=black!40, draw=black!40] (0+\i*\spc, 0+0*\spc) -- (\side+\i*\spc, 0+0*\spc) -- (\side+\i*\spc, \side/4+0*\spc) -- (0+\i*\spc, \side/4+0*\spc) -- cycle; 
                    \draw[fill=black!40, draw=black!40] (0+\i*\spc, 3*\side/4+0*\spc) -- (\side+\i*\spc, 3*\side/4+0*\spc) -- (\side+\i*\spc, \side+0*\spc) -- (0+\i*\spc, \side+0*\spc) -- cycle; 
                }

                \foreach \i in {0, 1, 2}{
                    \foreach \j in {0, 1}{
                        \draw (0+\i*\spc, 0+\j*\spc) -- (\side+\i*\spc, 0+\j*\spc) -- (\side+\i*\spc, \side+\j*\spc) -- (0+\i*\spc, \side+\j*\spc) -- cycle; 

                        \draw[black!50] (\side/2+\i*\spc, 0+\j*\spc) -- (\side/2+\i*\spc, \side+\j*\spc); 
                        \draw[black!50](0+\i*\spc, \side/2+\j*\spc) -- (\side+\i*\spc, \side/2+\j*\spc); 

                        \node[scale=\txt] at (0+\i*\spc, 0+\j*\spc-0.5) {$-\pi$}; 
                        \node[scale=\txt] at (\side/4+\i*\spc, 0+\j*\spc-0.5) {$-\frac{\pi}{2}$}; 
                        \node[scale=\txt] at (\side/2+\i*\spc, 0+\j*\spc-0.5) {$0$};
                        \node[scale=\txt] at (3*\side/4+\i*\spc, 0+\j*\spc-0.5) {$\frac{\pi}{2}$}; 
                        \node[scale=\txt] at (\side+\i*\spc, 0+\j*\spc-0.5) {$\pi$};

                        \node[scale=\txt] at (0+\i*\spc-0.7, 0+\j*\spc) {$-\pi$}; 
                        \node[scale=\txt] at (0+\i*\spc-0.7, \side/4+\j*\spc) {$-\frac{\pi}{2}$}; 
                        \node[scale=\txt] at (0+\i*\spc-0.5, \side/2+\j*\spc) {$0$};
                        \node[scale=\txt] at (0+\i*\spc-0.5, 3*\side/4+\j*\spc) {$\frac{\pi}{2}$}; 
                        \node[scale=\txt] at (0+\i*\spc-0.5, \side+\j*\spc) {$\pi$};
                    }
                }
            \end{tikzpicture}
        }
        \caption{High frequencies $\Theta^{high}$ (in grey). Top row: no coarsening in space. Bottom row: factor 2 coarsening in space. Left column: no coarsening in time. Middle column: factor 2 coarsening in time. Right column: factor 4 coarsening in time.}
        \label{fig:freqs}
    \end{figure}

\section{Analysis of the smoother}
\label{sec:smoother}

    In this section, we begin with the analysis for semi-coarsening in space and time, and we deduce the optimal damping parameter $\omega^*$ for full space-time coarsening as well as for our new coarsening strategy (see Figure~\ref{fig:coarsening-strategies}). Then, for the full space-time and new coarsening strategies, we compare the performance of the smoother with our optimized damping parameter to the one with the original choice $\omega^* = 0.5$ made in~\cite{gander2016analysis}. 
    

    To compute the optimal damping parameter, we first need to consider the smoothing factor defined as 
    \begin{equation}
        \label{eq:smoothing-factor}
        \mu_S(\omega) = \max\{|\hat{S}_{\tau, h}(\omega; \theta_{t}, \theta_{x})|; (\theta_t, \theta_x)\in \Theta^{high}\} \;,
    \end{equation}
    where $\Theta^{high}$ is the set of high frequencies defined in Section~\ref{sec:analysis-tech} and $\hat{S}_{\tau,h}$ is the Fourier symbol associated to~\eqref{eq:def-smoother}. 
    Thus, computing $\mu_S$ is equivalent to finding the modes $(\theta_{t}^{*}, \theta_{x}^{*})\in \Theta^{high}$ that degrade the smoothing factor the most, i.e., that maximize $|\hat{S}_{\tau, h}(\omega; \theta_x, \theta_t)|$.
    Now, given an expression of $\mu_S$, the optimal damping parameter is obtained by solving the optimization problem
    \begin{equation}
        \label{eq:min-smoothing-factor}
        \omega^* = \argmin_{\omega\in(0,1]} \mu_S(\omega)  \;.
    \end{equation}    

    
        
 
    In the case of a discretization using centered finite differences in space and a generic Runge Kutta method in time, the Fourier coefficient $\hat{S}_{\tau,h}$ of~\eqref{eq:def-smoother} is given by 
    \begin{equation}
    \label{eq:FT-smoother}
        \hat{S}_{\tau, h}(\omega; \theta_{t}, \theta_{x}) = 1 - \omega + \omega e^{-i\theta_t}R(z(\theta_x))\;,
    \end{equation}
    where $z(\theta_x)$ is the Fourier coefficient of the matrix $L_h$ in~\eqref{eq:heat-discretized-in-space} scaled by $\tau$.
    To ease computations, we consider
    \begin{multline} \label{eq:s-squared}
        |\hat{S}_{\tau, h}(\omega; \theta_{t}, \theta_{x})|^{2} = (1 - \omega)^2 + 2\omega(1 - \omega) \cos\theta_t R(z(\theta_x)) \\+ \omega^2 R^2(z(\theta_x))\;.  
    \end{multline}

    In particular, we will choose $R$ to be the stability function associated to the Backward Euler method, 
        \begin{equation}
        \label{eq:R-euler-imp}
        R(z) = \frac{1}{1-z}\;,
    \end{equation}
    and $z(\theta_x) = 2\tau[\cos\theta_x -1]/h^2 = 2\sigma[\cos\theta_x -1]$ (from a centered spatial discretization).
    
    Denote by $c_x := 1 - 2\sigma[\cos\theta_x - 1]$. Then, the expression~\eqref{eq:s-squared} can then explicitly be written as
    \begin{equation}
        \label{eq:s-squared-euler}
        |\hat{S}_{\tau, h}(\omega; \theta_{t}, \theta_{x})|^{2} = (1 - \omega)^{2} + \frac{2\omega(1 - \omega)\cos\theta_t}{c_x} + \frac{\omega^{2}}{c_x^{2}}\;.
    \end{equation}
    When $\theta_x=\frac{\pi}{2}$, we denote $c_x$ by
    \begin{equation*}
        c:= 1 - 2\sigma[\cos\frac{\pi}{2} - 1]=1+2\sigma\;.
    \end{equation*}
    
    \begin{remark}
        As the term~\eqref{eq:s-squared-euler} is symmetric with respect to $\theta_x$ (due to the symmetry of the spatial discretization) and $\theta_t$, we can restrict our arguments to the high frequencies that lie in the upper right quadrant. 
    \end{remark}

    \begin{lemma}\label{lemma:time-coarsening}
        Let $R$ be the stability function of a Runge-Kutta method such that $R(z)\geq 0$ for all $z\in \mathbb{R}$.
        For $\omega\in(0, 1]$, the time mode that degrades the smoother the most is
        \begin{equation*}
            \theta_t^* = \begin{cases}
                0 & \text{if no time coarsening,}\\ 
                \pi/2 & \text{if factor $2$ time coarsening,} \\ 
                \pi/4 & \text{if factor $4$ time coarsening.}
            \end{cases}
        \end{equation*}
    \end{lemma}
    \begin{proof}
        The term~\eqref{eq:s-squared} can be separated into three terms, the first and third do not depend on $\theta_t$ and are positive. We only need to maximize the second term 
        \begin{equation*}
            2\omega (1 - \omega)\cos\theta_t\,R(z(\theta_x)) \;.
        \end{equation*} 
        As $R(z)\geq 0$, we only need to focus on maximizing $\cos\theta_t$. For $\theta_t\in[0, \pi]$ (no coarsening), it is obtained for $\theta_t^*=0$. For $\theta_t\in[\pi/2, \pi]$ (factor $2$ coarsening in time), we have that $\theta_t^*=\pi/2$. For $\theta_t\in[\pi/4, \pi]$ (factor $4$ coarsening in time), we have that $\theta_t^*=\pi/4$.
    \end{proof}

    \begin{lemma}\label{lemma:space-coarsening}
        Let $R$ be the stability function of a Runge-Kutta method such that $R$ is a monotonically increasing function.
        For $\omega\in(0, 1]$, the spatial mode that degrades the smoother the most is the one that maximizes $z(\theta_x)$.
    \end{lemma}
    \begin{proof}
        Proceeding as previously, we look at the terms in~\eqref{eq:s-squared}, the first one is a constant, the second one depends on $\theta_t$ and $\theta_x$ and the third one depends on $\theta_x$. 
        By Lemma~\ref{lemma:time-coarsening}, $\cos\theta_t^*\geq 0$, thus in order to maximize the term~\eqref{eq:s-squared}, the term $R(z(\theta_x))$ needs to be maximized. Since $R$ is a monotonically increasing function, we need to maximize $z(\theta_x)$.
    \end{proof}
    
    


    \begin{lemma}\label{lemma:optimal-modes}
        Let us consider the Backward Euler method whose stability function is given by \eqref{eq:R-euler-imp}. Then we are able to compute the modes $(\theta_t^*, \theta_x^*)$ for which the term~\eqref{eq:s-squared-euler} is the largest for each of the following coarsening strategies.
        \begin{enumerate}        
            \item[(i)] When considering semi-coarsening in time,
            \begin{equation*}
                (\theta_t^*, \theta_x^*) = (\pi/2,0) \;.   
            \end{equation*}
            \item[(ii)] When considering semi-coarsening in space,
            \begin{equation*}
                (\theta_t^*, \theta_x^*) = (0,\pi/2) \;.   
            \end{equation*}
            \item[(iii)] When considering full space-time coarsening, 
            \begin{equation*}
                (\theta_t^*, \theta_x^*) = \begin{cases}
                    (0,\pi/2) & \text{if} \quad (c,\omega) \in \mathcal{D}_1,\\ 
                    (\pi/2,0) & \text{otherwise,}
                \end{cases}
            \end{equation*}
            where the region $\mathcal{D}_1\subset (1,+\infty)\times (0,1]$ is defined by
            \begin{equation*}
                \mathcal{D}_1 := \left\{ c\in [1,+\infty), \: 0 < \omega \leq \frac{2c}{c^2 + 2c - 1} \right\}\;.
            \end{equation*}
            \item[(iv)] When considering the new coarsening strategy, 
            \begin{equation*}
                (\theta_t^*, \theta_x^*) = \begin{cases}
                    (0,\pi/2) & \text{if} \quad (c,\omega) \in \mathcal{D}_2,\\ 
                    (\pi/4,0) & \text{otherwise,}
                \end{cases}
            \end{equation*}
            where the region $\mathcal{D}_2\subset (1,+\infty)\times (0,1]$ is defined by
            \begin{equation*}
                \mathcal{D}_2 := \left\{ c\in [1,\sqrt2], \:0 < \omega \leq \frac{\sqrt{2}c^2-2c}{(\sqrt2-1)c^2 - 2c + 1} \right\}\;.
            \end{equation*}
        \end{enumerate}
        \end{lemma}

    \begin{proof}
        \textit{(i)} If no coarsening in space is used, the high frequency spatial modes are in the set $[0, \pi]$. Thus by Lemma~\ref{lemma:space-coarsening}, we need to maximize $z(\theta_x) = 2\sigma[\cos\theta_x - 1]$
        which is achieved for $\theta_x^* = 0$, and we get from Lemma~\ref{lemma:time-coarsening} that $\theta_x^* = \pi/2$. \\
        \textit{(ii)} In the case of semi-coarsening in space, this maximum is achieved for $\theta_x^* = \pi/2\in [\pi/2, \pi]$. Besides, Lemma~\ref{lemma:time-coarsening} yields $\theta_t^*=0$. \\
        \textit{(iii)} In the case of full space-time coarsening, we can see the domain of high frequencies as the union of the high frequencies associated to semi-coarsening in time and to semi-coarsening in space (see Figure~\ref{fig:coarsening-strategies}). Thus, depending on the values of $\omega$ and $\sigma$, the modes $(\theta_t^*, \theta_x^*)$ that degrade the smoother the most are either those for semi-coarsening in time $(\pi/2,0)$ (we say that \textit{time dominates}) or those for semi-coarsening in space $(0,\pi/2)$ (we say that \textit{space dominates}).
        More specifically, we have that space dominates for values of $\omega\in (0, 1]$ such that 
        \begin{equation*}
            |\hat{S}_{\tau, h}(\omega; \pi/2, 0)|^2 \leq |\hat{S}_{\tau, h}(\omega; 0, \pi/2)|^2\;.
        \end{equation*}
        Solving this inequality for $\omega$ and $\sigma$ leads to $(c,\omega)\in\mathcal{D}_1$.
        In Figure~\ref{fig:why-intersection} (top), this region $\mathcal{D}_1$ where space dominates is represented in orange, and the region where time dominates is represented in blue. The previous inequality becomes an equality when $\omega=\frac{2c}{c^2 + 2c - 1}$, see the grey curve in Figure~\ref{fig:why-intersection}. \\
        \textit{(iv)} For the new coarsening strategy, we need to compute the optimal modes for semi-coarsening in time with a factor of $4$. It is given by $(\theta_t^*, \theta_x^*)=(\pi/4, 0)$ using Lemmas~\ref{lemma:time-coarsening} and~\ref{lemma:space-coarsening}. So this time, we have that $(\theta_t^*, \theta_x^*)=(\pi/4, 0)$ when time dominates and $(0,\pi/2)$ when space dominates.
        As in the previous case, we know that space dominates for $\omega\in(0, 1]$ such that
        \begin{equation*}
            |\hat{S}_{\tau, h}(\omega; \pi/4, 0)|^2 \leq |\hat{S}_{\tau, h}(\omega; 0, \pi/2)|^2\;.
        \end{equation*}
        Here, the solution is given by $(c,\omega)\in\mathcal{D}_2$, see Figure~\ref{fig:why-intersection} (bottom), and some technical computations show that equality is reached for $\omega=\frac{\sqrt{2}c^2-2c}{(\sqrt2-1)c^2 - 2c + 1}$.
    \end{proof}

    \begin{figure}
        \centering
        \includegraphics[width=\columnwidth]{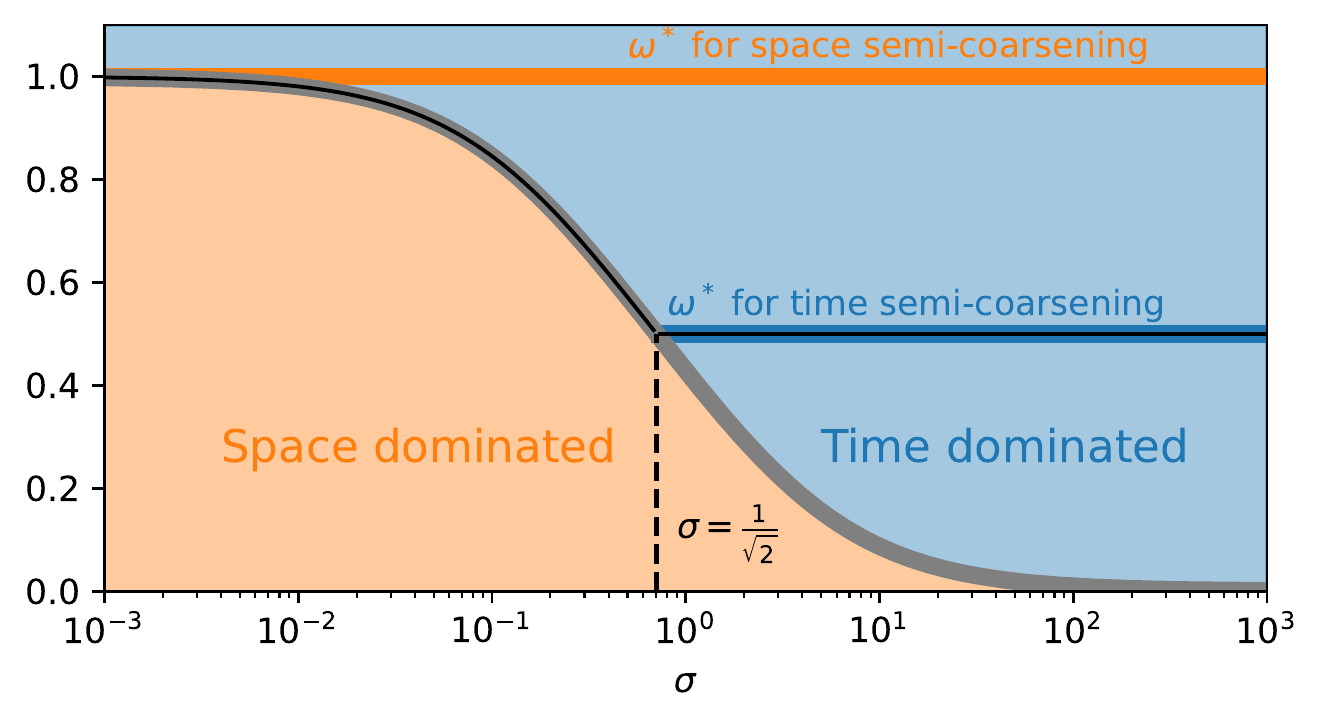}
        \includegraphics[width=\columnwidth]{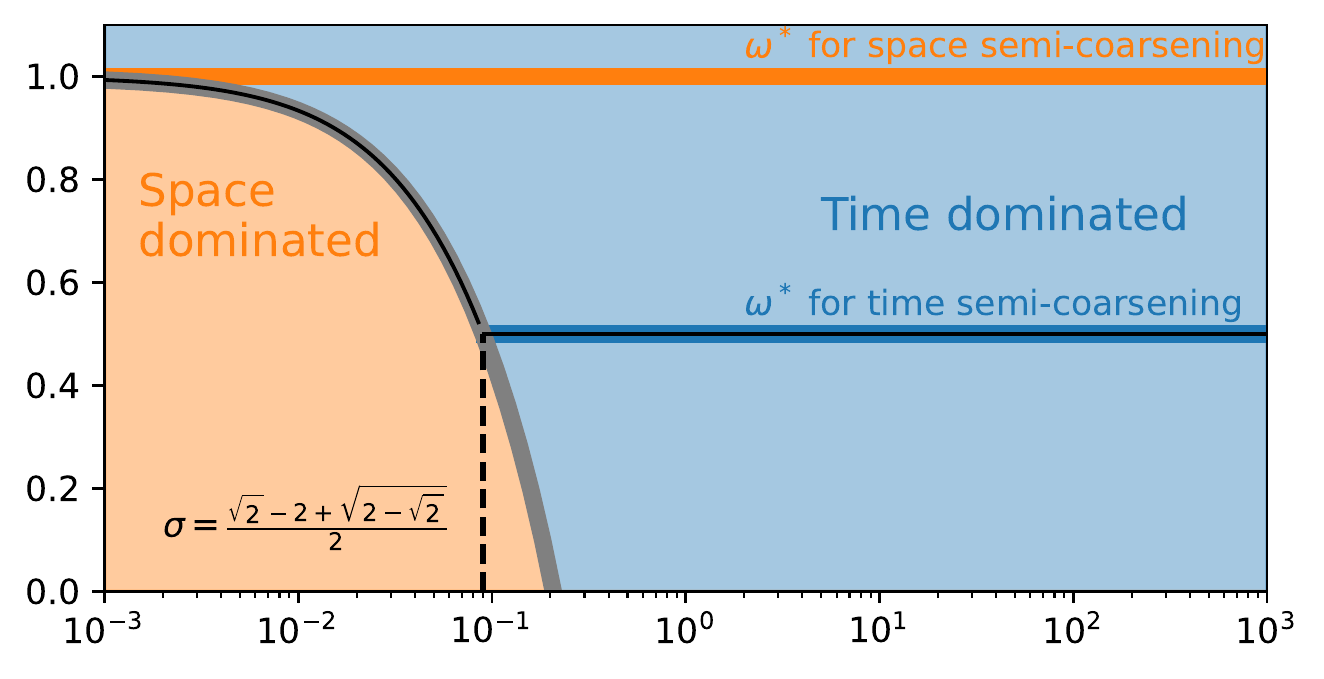}
        \caption{Illustration of how to choose $\omega^*$. Top: for the original strategy, bottom: for the new one. In orange and blue, we can observe the domains where space dominates over time and conversely. The continuous lines show the optimal value for each of the regimes. The regions $\mathcal{D}_{1}$ and $\mathcal{D}_{2}$ from Lemma~\ref{lemma:optimal-modes} correpond to regions in orange (where space dominates). }
        \label{fig:why-intersection}
    \end{figure}

    Now that we have found the pairs $(\theta_t^*,\theta_x^*)$ for each coarsening strategy, we are able to compute $\mu_S$ in each case. The next step is to figure out the optimal values for the damping parameter $\omega$ by solving \eqref{eq:min-smoothing-factor}.
    
    \begin{theorem}\label{thm:optimal-w}
        Consider the heat equation~\eqref{eq:continuous-heat-equation} discretized using a second order centered finite difference scheme in space and the Backward Euler method in time, which we solve with the Space-Time Multigrid algorithm. Then, we are able to compute the optimal damping parameter $\omega^*$ for the following coarsening strategies. 
        \begin{enumerate}
            \item[(i)] When considering  semi-coarsening in time, 
            \begin{equation*}
                \omega^{*} = \frac{1}{2}\;.
            \end{equation*}
            \item[(ii)] When considering semi-coarsening in space, 
            \begin{equation*}
                \omega^{*} = 1\;.
            \end{equation*}
            \item[(iii)] When considering full space-time coarsening, 
            \begin{equation*}
                \omega^{*} = 
                \begin{cases}
                    \frac{1}{2} & \text{if}\quad \sigma >\frac{1}{\sqrt{2}}\\
                    \frac{2c}{c^2+2c-1} & \text{otherwise.}
                \end{cases}
            \end{equation*}
            \item[(iv)] When considering the new coarsening strategy, 
            \begin{equation*}
                \omega^{*} = 
                \begin{cases}
                    \frac{1}{2} & \text{if}\quad \sigma > \frac{\sqrt2-2+\sqrt{2-\sqrt2}}{2}\\ 
                    \frac{\sqrt2 c^2-2c}{(\sqrt2-1)c^2 -2c +1}& \text{otherwise.}
                \end{cases}
            \end{equation*}
        \end{enumerate}
    \end{theorem}
    \begin{proof}
        \textit{(i)} For semi-coarsening in time, to obtain the optimal values $\omega^*$, we consider $\mu_S$ as defined in~\eqref{eq:smoothing-factor}, then compute $d\mu_S/d\omega$ and solve 
        $d\mu_S/d\omega = 0$. 
        If we coarsen by a factor $2$, 
        \begin{equation*}
            \frac{d\mu_S}{d\omega}(\omega) = -2 + 4\omega = 0 \;,
        \end{equation*}
        where solving for $\omega$ yields $\omega^* = 1/2$.
        Now, if we coarsen by a factor $4$, 
        \begin{equation*}
            \frac{d\mu_S}{d\omega}(\omega) = 2\omega(2 - \sqrt{2}) - (2 - \sqrt{2}) = 0\;,
        \end{equation*}
        where solving for $\omega$ yields again $\omega^* = 1/2$. \\
        %
        \textit{(ii)} In the case of semi-coarsening in space, we have
        \begin{equation*}
            \mu_S(\omega) = |\hat{S}_{\tau,h}(\omega; 0, \pi/2)|^2 = \big(1 + (c^{-1} -1)\omega\big)^2\;.
        \end{equation*}  
        We notice that $\mu_S$ is a parabola with a minimum at $\omega_0 = \frac{c}{c-1}$, which is outside the domain $(0, 1]$. Therefore it is a decreasing function of $\omega$ on that domain, and it follows that $\omega^* = 1$. \\
        \textit{(iii)} In the case of full space-time coarsening, we can observe two regimes (see Lemma~\ref{lemma:optimal-modes} and Figure~\ref{fig:why-intersection}): one where time dominates, and the other one where space dominates. 
        Also, we know that $\omega \mapsto|\hat{S}_{\tau,h}(\omega; 0, \pi/2)|^2$ and $\omega\mapsto|\hat{S}_{\tau,h}(\omega; \pi/2, 0)|^2$ are parabolas attaining their minima on $(0, 1]$ at $\omega_0=1$ and $\omega_0=1/2$.
        Therefore, for a given $\sigma$, there are two possible behaviours for the continuous function $\omega\mapsto\mu_S(\omega)$, see Figure~\ref{fig:why-intersection} (top). First, if $\sigma> 1/\sqrt2$, then $\mu_S$ is decreasing on $[0,1/2]$ and increasing on $[1/2,1]$, thus $\omega^*=1/2$. Second, if $\sigma< 1/\sqrt2$, then $\mu_S$ is decreasing on $[0,\frac{2c}{c^2+2c-1}]$ and increasing on $[\frac{2c}{c^2+2c-1},1]$, thus $\omega^*=\frac{2c}{c^2+2c-1}$. \\
        %
        \textit{(iv)} In the case of the new coarsening strategy, we proceed exactly in the same way since we know that $\omega \mapsto|\hat{S}_{\tau,h}(\omega; 0, \pi/4)|^2$ is also a parabola attaining its minimum at $\omega_0=1/2$. Here again, for a given $\sigma$, there are two distinct cases, see Figure~\ref{fig:why-intersection} (bottom). First, if $\sigma>\frac{\sqrt2-2+\sqrt{2-\sqrt2}}{2}$, then $\mu_S$ is decreasing on the interval $[0,1/2]$ and increasing on the interval $[1/2,1]$, thus $\omega^*=1/2$. Second, if $\sigma< \frac{\sqrt2-2+\sqrt{2-\sqrt2}}{2}$, then $\mu_S$ is decreasing on $[0,\frac{\sqrt2 c^2-2c}{(\sqrt2-1)c^2 -2c +1}]$ and increasing on $[\frac{\sqrt2 c^2-2c}{(\sqrt2-1)c^2 -2c +1},1]$, thus $\omega^*=\frac{\sqrt2 c^2-2c}{(\sqrt2-1)c^2 -2c +1}$. 
        %
    \end{proof}

    \begin{remark}
        It can be checked that the optimal damping parameters for full space-time coarsening as well as for the new strategy belong to $[1/2, 1]$ for $\sigma\in[0, 1/\sqrt{2}]$ and $\sigma\in[0, \frac{\sqrt{2}-2 + \sqrt{2 - \sqrt{2}}}{2}]$ respectively. 
    \end{remark}

    Figures~\ref{fig:optimal-vs-neumuller} and~\ref{fig:optimal-vs-chaudet} show a comparison between the smoothing factors corresponding to the choice of the damping parameter in~\cite{gander2016analysis} ($\omega=1/2$) and the ones proposed in Theorem~\ref{thm:optimal-w} ($\omega=\omega^*$) for the original and the new coarsening strategies. 
    In each case, as expected, choosing $\omega=\omega^*$ leads to a smaller smoothing factor, although the gain does not seem very important for the new strategy.
    In order to better illustrate this gain, we plotted in Figure~\ref{fig:efficiency} the ratio $\frac{\ln[\mu_S(\omega^*)]}{\ln[\mu_S(0.5)]}$ as a function of $\sigma$, which we call the \textit{efficiency}.
    This quantity is an approximation of the ratio $\frac{N^{default}}{N^{optimal}}$ where $N^{default}$ and $N^{optimal}$ are the numbers of smoothing iterations it takes to reach a given error if we take $\omega=1/2$ and $\omega=\omega^*$.
    In Figure~\ref{fig:efficiency}, we see that considering an optimal damping parameter can give up to a factor $2$ efficiency with respect to the value proposed in~\cite{gander2016analysis}. This means that an iteration using the optimal damping parameter would give a similar smoothing of the solution as two iterations using the original one. 
    Additionally, using an optimal damping parameter does not give any overhead: in any case, the value of $\sigma$ needs to be checked at each level to decide if semi-coarsening in time or full space-time coarsening is more appropriate.
    
    \begin{figure}
        \centering
        \begin{subfigure}[t]{\columnwidth}
            \centering
            \includegraphics[width=\textwidth]{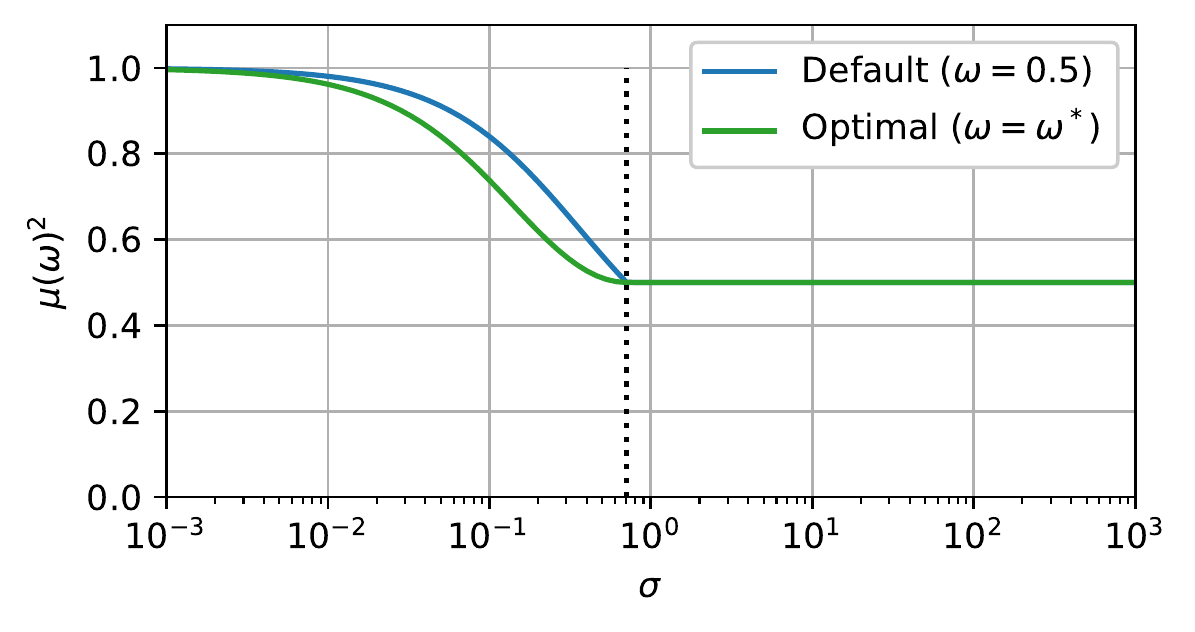}
            \caption{Smoothing factor for the original coarsening strategy}
            \label{fig:optimal-vs-neumuller}
        \end{subfigure}
        \begin{subfigure}[t]{\columnwidth}
            \centering
            \includegraphics[width=\textwidth]{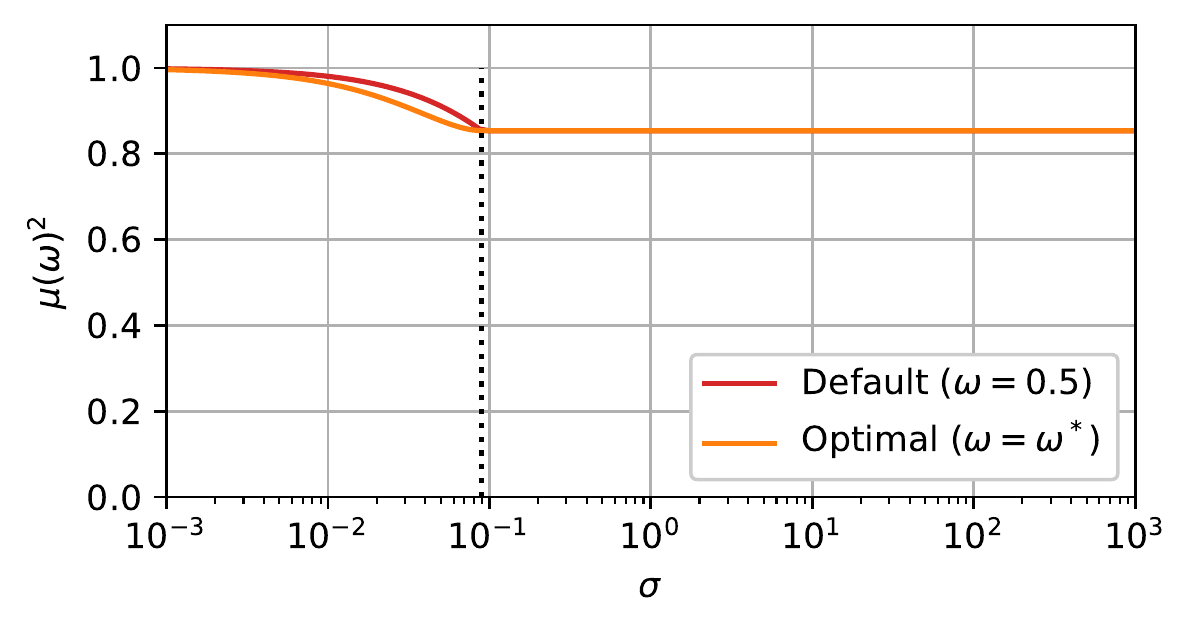}
            \caption{Smoothing factor for the new coarsening strategy}
            \label{fig:optimal-vs-chaudet}
        \end{subfigure}
        \begin{subfigure}[t]{\columnwidth}
            \centering
            \includegraphics[width=\textwidth]{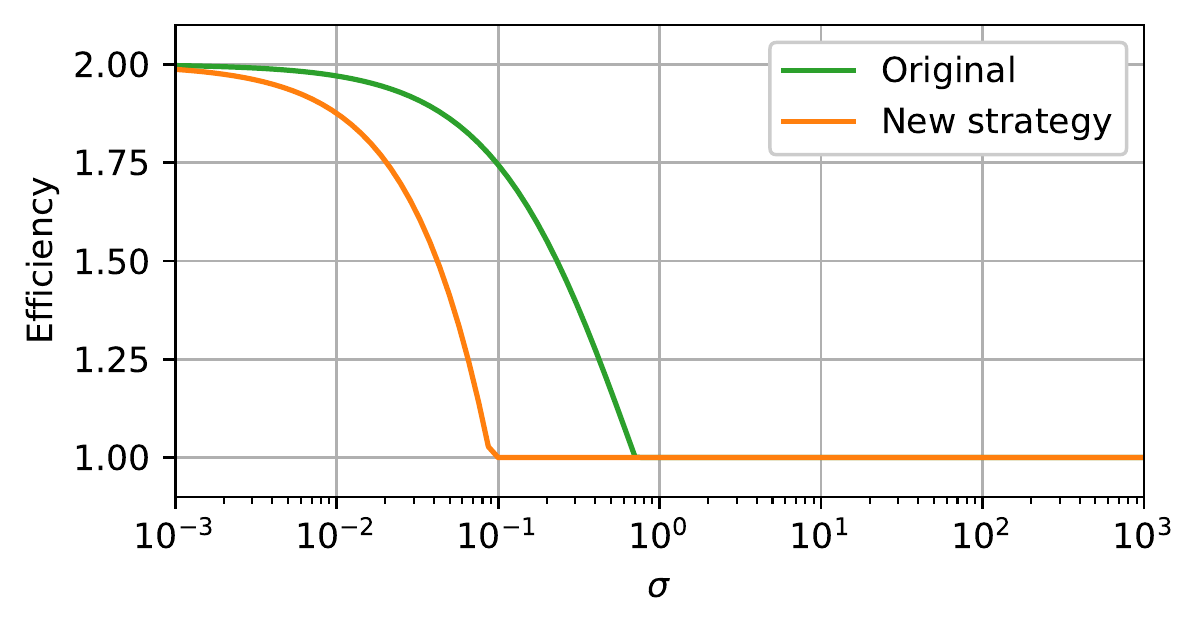}
            \caption{Relative efficiency of the smoother, defined as $\frac{\ln[\mu_S(\omega^*)]}{\ln[\mu_S(0.5)]}$}
            \label{fig:efficiency}
        \end{subfigure}
        
        \caption{Top and middle: smoothing factors for both coarsening strategies with the optimal smoothing parameter $\omega^*$ as given by Theorem~\ref{thm:optimal-w} and with $\omega=0.5$ as suggested by~\cite{gander2016analysis}.
        Bottom: efficiency of the optimized smoother compared to the one proposed in~\cite{gander2016analysis}.}
        \label{fig:optimal-vs-default}
    \end{figure}

\begin{remark}
    Taking a look at the smoothing factors displayed in Figures~\ref{fig:optimal-vs-neumuller} and \ref{fig:optimal-vs-chaudet}, one may think that the full space-time coarsening strategy is much more efficient than the new coarsening strategy. Actually, comparing these values is not really fair as the two strategies do not correspond to the same grids. Indeed, the first one is for coarsening from $(\tau,h)$ to $(2\tau,2h)$ while the second one is for coarsening from $(\tau,h)$ to $(4\tau,2h)$. Therefore $\mu_S$ in Figure~\ref{fig:optimal-vs-chaudet} is computed by taking the maximum over a set of high frequencies $\Theta^{high}$ larger than the one in the case of Figure~\ref{fig:optimal-vs-neumuller}, which explains why the values are significantly higher. 
\end{remark}

\section{Comparison of different multigrid strategies}
\label{sec:coarse-grid-analysis}

In this section, we are interested in the study of Space-Time Multigrid strategies for solving \eqref{eq:continuous-heat-equation}. 
The most natural strategy would be to build a multigrid cycle based on full space-time coarsening. This choice does not alter the geometric properties of the mesh and the gain obtained by each coarsening step is significant in terms of computational cost since both space and time dimensions are divided by two. However, if the ratio $\sigma$ is given by $\sigma^0:=\tau/h^2$ at the finest level, then at the second level it becomes $\sigma=\sigma^0/2$, and at the $n$-th level we end up with $\sigma=\sigma^0/2^{n-1}$.
As illustrated in Figures~\ref{fig:why-intersection} and~\ref{fig:optimal-vs-default}, the smoothing factor $\mu_S$ tends to deteriorate when $\sigma$ gets too small. Thus, if we use too many levels with this strategy, we will end up with very inefficient smoothing steps, which might slow down the whole multigrid algorithm.

For the reasons mentioned above, we thus want to study strategies such that the ratio $\sigma$ remains close to its value $\sigma^0$ at the finest level along all different levels. Especially, we focus on the two strategies presented in Figure~\ref{fig:coarsening-strategies}. In the first one (called \textit{original strategy}), inspired by the alternating strategy introduced by Neumüller and Gander in~\cite[Section 5]{gander2016analysis}, the simplest V-cycle consists of one full space-time coarsening, then one semi-coarsening in time. This leads to a three-level method, starting from the level $(\tau,h)$ and going down to the coarsest level $(4\tau,2h)$. In the second strategy (called \textit{new strategy}), which we present for the first time in this manuscript, the simplest V-cycle consists of a one step coarsening from the level $(\tau,h)$ directly to the level $(4\tau,2h)$. Therefore it leads to a standard two-level method.

In what follows, we will derive the expressions of the multigrid iteration operators associated with these different strategies, and we will estimate their spectral radii in the case where the damping parameters of the smoothers are chosen according to Theorem~\ref{thm:optimal-w}. Then, we will be interested in the action of these multigrid operators on low frequency modes. Finally, we will try to find damping parameters that optimize the whole multi-level method instead of optimizing only the smoother.

\subsection{Original strategy}
\label{sec:strategie-neumuller}

For the simplest three-level V-cycle illustrated in Figure~\ref{fig:coarsening-strategies} (left), we consider $\nu_1$ and $\nu_2$ pre/post-smoothing steps at the finest level and $\eta_1$ and $\eta_2$ pre/post-smoothing steps at the level $(2\tau,2h)$.

$\:$\\
\textbf{Notation} The operator associated with the space-time prolongation from a given coarse grid $(m\tau',nh')$ to the finer grid $(\tau',h')$ will be denoted by $P_{m,n}$, for $m,n\in\{1,2\}$. This notation does not take into account the grid $(\tau',h')$ considered therefore the definition of $P_{m,n}$ depends on the context, e.g. for $m=n=2$:
\begin{equation*}
    \begin{aligned}
        &P_{2,2}L_{2\tau,2h} = (P_{2\tau\to\tau}\otimes P_{2h\to h})L_{2\tau,2h}\;, \\
        &P_{2,2}L_{4\tau,2h} = (P_{4\tau\to 2\tau}\otimes P_{2h\to h})L_{4\tau,2h}\;.
    \end{aligned}  
\end{equation*}
Similarly, we denote the corresponding restriction operator from the grid $(\tau',h')$ to the coarser grid $(m\tau',nh')$ by $R_{m,n}=\frac{1}{2}P_{m,n}^T$. \\

Using the notations introduced above, we can write the expression of the multigrid iteration operator as
\begin{equation}
    M^o = S_{\tau,h}^{\nu_{2}}\big[I_{\tau,h} - P_{2,2}(\widetilde{L_{2\tau,2h}^{-1}})R_{2,2}L_{\tau,h}\big]S_{\tau,h}^{\nu_{1}} \;,
\end{equation}
where $\widetilde{L_{2\tau,2h}^{-1}}$ is an approximation of the inverse of $L_{2\tau,2h}$ given by
\begin{equation}
\label{eq:def-app-Lcoarseinv}
    \begin{aligned}
        \widetilde{L_{2\tau,2h}^{-1}} := &\left( I_{2\tau,2h} - S_{2\tau,2h}^{\eta_{2}}\big[I_{2\tau,2h} \right. \\
        & \left. - P_{2,1}(L_{4\tau,2h}^{-1})R_{2,1}L_{2\tau,2h}\big]S_{2\tau,2h}^{\eta_{1}} \right) L^{-1}_{2\tau,2h} \;.
    \end{aligned}
\end{equation}

\begin{remark}
    The definition of $\widetilde{L_{2\tau,2h}^{-1}}$ reflects the fact that, instead of solving exactly on the grid $(2\tau,2h)$ with $L_{2\tau,2h}^{-1}$, we solve approximately using one V-cycle on a coarser grid starting with zero initial approximation.
\end{remark}

In order to study the properties of the operator $M^o$, we compute its action on the Fourier modes $\boldsymbol{\Phi}(\theta_t,\theta_x)$, for all $(\theta_t,\theta_x)\in\Theta$.

\paragraph{Preliminary notions.}
Unlike the smoothing operator, the multigrid iteration operator cannot be diagonalized by the Fourier modes. 
Indeed, for a given fine grid low frequency mode, when working on the coarse grid, it is not possible to distinguish this mode from a corresponding fine grid high frequency mode (often referred to as \textit{companion} mode). This phenomenon is called \textit{aliasing}. In terms of operator, when studying the action of $M^o$ on the fine grid modes, the aliasing phenomenon results in a coupling between low frequency modes and their high frequency companion modes. However, as we will show next, it is still possible to block diagonalize $M^o$ by gathering the fine grid modes in groups of companion modes, called \textit{spaces of harmonics}.
In order to explain this in detail, let us begin with introducing two mapping operators 
\begin{definition}
\label{def:mapping-freq-op}
    Let $\gamma_2:(-\frac{\pi}{2},\frac{\pi}{2}]\to(-\pi,-\frac{\pi}{2}]\cup(\frac{\pi}{2},\pi]$ and $\gamma_4:(-\frac{\pi}{4},\frac{\pi}{4}]\to(-\frac{\pi}{2},-\frac{\pi}{4}]\cup(\frac{\pi}{4},\frac{\pi}{2}]$ be defined as
    \begin{equation*}
        \begin{aligned}
            \gamma_2(\theta) &:= \theta-\sign(\theta)\pi\;, & \forall \theta \in (-\frac{\pi}{2},\frac{\pi}{2}]\;, \\
            \gamma_4(\theta) &:= \theta-\sign(\theta)\frac{\pi}{2}\;, & \forall \theta \in (-\frac{\pi}{4},\frac{\pi}{4}]\;,
        \end{aligned}
    \end{equation*}
    where we have chosen the convention $\sign(0)=-1$.
\end{definition}

For the reader's convenience, each low time frequency in $(-\frac{\pi}{4},\frac{\pi}{4}]$ will be denoted by $\check{\check{\theta}}_t$, and each low space frequency in $(-\frac{\pi}{2},\frac{\pi}{2}]$ will be denoted by $\check{\theta}_x$. Then, for each low time frequency, we will denote by $\hat{\check{\theta}}_t:=\gamma_2(\check{\check{\theta}}_t)$, $\check{\hat{\theta}}_t:=\gamma_4(\check{\check{\theta}}_t)$ and $\hat{\hat{\theta}}_t:=\gamma_2\circ\gamma_4(\check{\check{\theta}}_t)$. Using the same notation for space frequencies, we get $\hat{\theta}_x:=\gamma_2(\check{\theta}_x)$.

In the present section, the coarsest grid in the cycles studied is $(4\tau,2h)$, which means that the low frequency domain, corresponding to the frequencies that can be represented on this coarse grid, is given by $\Theta^{low}=(-\pi/4,\pi/4]\times(-\pi/2,\pi/2]$. Using the frequency mapping operators defined above, it is possible to rewrite the Fourier decomposition \eqref{eq:Fourier-decomp} as a sum on the low frequencies only.

\begin{lemma}
    Let $\boldsymbol{u}=(\boldsymbol{u}_1, \boldsymbol{u}_2,\dots, \boldsymbol{u}_{N_t})\in\mathbb{R}^{N_t N_x}$. Then the vector $\boldsymbol{u}$ can be written as
    \begin{equation}
        \label{eq:Fourier-decomp-bis}
        \begin{aligned}
            \boldsymbol{u} = \hspace{-1em} & \sum_{(\check{\check{\theta}}_t,\check{\theta}_x)\in\Theta^{low}} \hspace{-1em} \left[\boldsymbol{\Psi}(\check{\check{\theta}}_t,\check{\theta}_x) + \boldsymbol{\Psi}(\check{\hat{\theta}}_t,\check{\theta}_x) + \boldsymbol{\Psi}(\hat{\check{\theta}}_t,\check{\theta}_x) + \boldsymbol{\Psi}(\hat{\hat{\theta}}_t,\check{\theta}_x) \right. \\
            & \quad + \boldsymbol{\Psi}(\check{\check{\theta}}_t,\hat{\theta}_x) + \boldsymbol{\Psi}(\check{\hat{\theta}}_t,\hat{\theta}_x) + \boldsymbol{\Psi}(\hat{\check{\theta}}_t,\hat{\theta}_x) + \boldsymbol{\Psi}(\hat{\hat{\theta}}_t,\hat{\theta}_x) \left. \right],
        \end{aligned}
    \end{equation} 
    where we have used the notation
    \begin{equation*}
        \boldsymbol{\Psi}(\theta_t,\theta_x) := \hat{u}(\theta_t,\theta_x)\boldsymbol{\Phi}(\theta_t,\theta_x).
    \end{equation*}
\end{lemma}
\begin{proof}
    The proof simply relies on the fact that the operators $\gamma_2:(-\frac{\pi}{2},\frac{\pi}{2}]\to(-\pi,-\frac{\pi}{2}]\cup(\frac{\pi}{2},\pi]$ and $\gamma_4:(-\frac{\pi}{4},\frac{\pi}{4}]\to(-\frac{\pi}{2},-\frac{\pi}{4}]\cup(\frac{\pi}{4},\frac{\pi}{2}]$ are one-to-one.
\end{proof}
This result means that each vector $u$ can be expressed as a linear combination, over all $(\check{\check{\theta}}_t,\check{\theta}_x)\in\Theta^{low}$, of elements in the \textit{space of harmonics} associated to the fine grid defined by
\begin{equation}
\label{eq:companion-modes}
    \begin{aligned}
        \mathcal{E}_{\tau,h}(\theta_t,\theta_x) := \text{span} \left\{ \right. & \left. \boldsymbol{\Phi}(\check{\check{\theta}}_t,\check{\theta}_x), 
        \boldsymbol{\Phi}(\check{\hat{\theta}}_t,\check{\theta}_x),
        \boldsymbol{\Phi}(\hat{\check{\theta}}_t,\check{\theta}_x),  \right. \\
        & \left. \boldsymbol{\Phi}(\hat{\hat{\theta}}_t,\check{\theta}_x),
        \boldsymbol{\Phi}(\check{\check{\theta}}_t,\hat{\theta}_x),
        \boldsymbol{\Phi}(\check{\hat{\theta}}_t,\hat{\theta}_x), \right. \\
        &\left. \boldsymbol{\Phi}(\hat{\check{\theta}}_t,\hat{\theta}_x),  
         \boldsymbol{\Phi}(\hat{\hat{\theta}}_t,\hat{\theta}_x)\: \right\} .
    \end{aligned}
\end{equation}
For simplicity, we have chosen the notation $(\theta_t,\theta_x)$ to refer to the low frequency $(\check{\check{\theta}}_t,\check{\theta}_x)$.
In the same spirit, we define the Fourier spaces associated with the coarser grids $(2\tau,2h)$ and $(4\tau,2h)$ as
\begin{equation*}
    \begin{aligned}
        &\mathcal{F}_{2\tau,2h}(2\theta_t,2\theta_x):= \text{span} \left\{ \boldsymbol{\Phi}(2\check{\theta}_t,2\theta_x), \boldsymbol{\Phi}(2\hat{\theta}_{t},2\theta_{x}) \right\}, \\
        &\mathcal{F}_{4\tau,2h}(4\theta_t,2\theta_x):= \text{span} \left\{ \boldsymbol{\Phi}(4\theta_t,2\theta_x)\right\}.
    \end{aligned}
\end{equation*}

Especially, for each low frequency $(\theta_t,\theta_x)$, the eight fine grid modes spanning the space of harmonics get aliased to the same coarse grid mode $\boldsymbol{\Phi}(4\theta_t,2\theta_x)$ spanning the Fourier space associated to the coarsest grid. Therefore, in order to study the convergence properties of the operator $M^o$, we will focus on its action on these eight fine grid modes for each low frequency.

\paragraph{Action of $S_{\tau,h}$.}
For each mode $\boldsymbol{\Phi}(\theta_t,\theta_x)$, we have already seen in \eqref{eq:FT-smoother} that
\begin{equation}
\label{eq:smoother-symbol}
    \begin{aligned}
        S_{\tau,h}\boldsymbol{\Phi}(\theta_t,\theta_x) &= \left(1-\omega +\frac{\omega}{c_x}e^{-i\theta_t} \right)\boldsymbol{\Phi}(\theta_t,\theta_x) \\
        &=: \hat{S}_{\tau,h}(\theta_t,\theta_x)\boldsymbol{\Phi}(\theta_t,\theta_x)\;.
    \end{aligned}
\end{equation}
For the reader's convenience, since $\omega$ is fixed here, we have dropped the dependence on $\omega$ in $\hat{S}_{\tau,h}$.
Now, we can deduce from~\eqref{eq:smoother-symbol} that the smoother maps the space of harmonics to itself. Furthermore, the action of the smoother on this space of harmonics takes the form of a diagonal matrix in $\mathbb{C}^{8\times 8}$:
\begin{equation*}
    \widehat{S_{\tau,h}} : \mathcal{E}_{\tau,h}(\theta_t,\theta_x) \to \mathcal{E}_{\tau,h}(\theta_t,\theta_x)\;,
\end{equation*}
\begin{equation*}
    \widehat{S_{\tau,h}}:=
    \begin{pmatrix}
            \hat{S}_{\tau,h}(\check{\check{\theta}}_t,\check{\theta}_x) \\ 
            & \hat{S}_{\tau,h}(\check{\check{\theta}}_t,\hat{\theta}_x)\\ 
            && \ddots \\ 
            &&& \hat{S}_{\tau,h}(\hat{\hat{\theta}}_t,\check{\theta}_x)
    \end{pmatrix}
    .
\end{equation*}

\paragraph{Action of $L_{\tau,h}$.}
For each mode $\boldsymbol{\Phi}(\theta_t,\theta_x)$, we have
\begin{equation}
    \label{eq:action-L-fine}
    \begin{aligned}
        L_{\tau,h}\boldsymbol{\Phi}(\theta_t,\theta_x) &= \left(1-e^{-i\theta_t} +2\sigma(1-\cos\theta_x) \right)\boldsymbol{\Phi}(\theta_t,\theta_x) \\
        &=: \hat{L}_{\tau,h}(\theta_t,\theta_x)\boldsymbol{\Phi}(\theta_t,\theta_x)\;.
    \end{aligned}
\end{equation}
This yields the following operator acting on the space of harmonics:
\begin{equation*}
    \begin{aligned}
        \widehat{L_{\tau,h}} &: \mathcal{E}_{\tau,h}(\theta_t,\theta_x) \to \mathcal{E}_{\tau,h}(\theta_t,\theta_x)\;, \\
        \widehat{L_{\tau,h}} &:=
        \begin{pmatrix}
            \hat{L}_{\tau,h}(\check{\check{\theta}}_t,\check{\theta}_x) \\ 
            & \hat{L}_{\tau,h}(\check{\check{\theta}}_t,\hat{\theta}_x)\\ 
            && \ddots \\ 
            &&& \hat{L}_{\tau,h}(\hat{\hat{\theta}}_t,\check{\theta}_x)
        \end{pmatrix}
        .
    \end{aligned}
\end{equation*}

\paragraph{Action of $R_{2,2}$.}
In the context of the restriction from the fine grid $(\tau,h)$, for each mode $\boldsymbol{\Phi}(\theta_t,\theta_x)$, we have
\begin{equation}
    \begin{aligned}
        R_{2,2}\boldsymbol{\Phi}(\theta_t,\theta_x) &= \frac{(1+\cos\theta_x)}{2}\frac{(1+\cos\theta_t)}{2}\boldsymbol{\Phi}(\theta_t,\theta_x) \\
        &=: \hat{R}(\theta_x)\hat{R}(\theta_t)\boldsymbol{\Phi}(2\theta_t,2\theta_x)\;.
    \end{aligned}
\end{equation}
Since $2\gamma_{2}(\theta)\equiv 2\theta$ for all $\theta\in(-\pi/2,\pi/2]$, it is clear from the previous equation that the first four modes in \eqref{eq:companion-modes} will be turned into the same ``coarser'' mode $\boldsymbol{\Phi}(2\check{\theta}_t,2\theta_x)$. Similarly, the last four modes in \eqref{eq:companion-modes} will be turned into $\boldsymbol{\Phi}(2\hat{\theta}_{t},2\theta_x)$. In matrix form, this takes the form of the operator
\begin{equation*}
    \begin{aligned}
        \widehat{R_{2,2}} &: \mathcal{E}_{\tau,h}(\theta_t,\theta_x) \to \mathcal{F}_{2\tau,2h}(2\theta_t,2\theta_x)\;, \\
        \widehat{R_{2,2}} &:=
        \begin{pmatrix}
            \hat{R}(\check{\check{\theta}}_t)\hat{R}(\check{\theta}_x) & \hat{R}(\check{\check{\theta}}_t)\hat{R}(\hat{\theta}_{x}) & \: \cdots \:\:  & 0 \\ 
            0 & 0 & \: \cdots \:\: & \hat{R}(\hat{\hat{\theta}}_t)\hat{R}(\hat{\theta}_{x})
        \end{pmatrix}
        .
    \end{aligned}
\end{equation*}

\paragraph{Action of $P_{2,2}$.}
Since the prolongation operator can be written as $P_{2,2}= 2R_{2,2}^T$, we deduce that its action on the Fourier space associated with the grid $(2\tau,2h)$ can be modeled by the matrix operator
\begin{equation*}
    \begin{aligned}
        \widehat{P_{2,2}} &: \mathcal{F}_{2\tau,2h}(2\theta_t,2\theta_x) \to \mathcal{E}_{\tau,h}(\theta_t,\theta_x)\;, \\
        \widehat{P_{2,2}} &:= 2 (\widehat{R_{2,2}})^T\;.
    \end{aligned}
\end{equation*}

The next step is to express in a similar way the action of the approximate inverse $(\widetilde{L_{2\tau,2h}^{-1}})$. In order to do this, we first need to study the action of the operators involved in its expression.

\paragraph{Action of $L_{2\tau,2h}$.}
In the case where we consider the rediscretization approach rather than the Galerkin approach, we have $L_{2\tau, 2h} = I_{2\tau}\otimes Q_{2\tau,2h}+\Gamma_{-1,2\tau}\otimes B_{2\tau,2h}$. Therefore we obtain the following operator acting on the Fourier space\MG{:}
\begin{equation*}
    \begin{aligned}
        \widehat{L_{2\tau,2h}} &: \mathcal{F}_{2\tau,2h}(2\theta_t,2\theta_x) \to \mathcal{F}_{2\tau,2h}(2\theta_t,2\theta_x)\;, \\
        \widehat{L_{2\tau,2h}} &:=
        \begin{pmatrix}
            \hat{L}_{2\tau,2h}(2\check{\theta}_t,2\theta_x) & 0 \\ 
            0 & \hat{L}_{2\tau,2h}(2\hat{\theta}_{t},2\theta_{x})
        \end{pmatrix}
        ,
    \end{aligned}    
\end{equation*}
where we have used the notation
\begin{equation*}
    \hat{L}_{2\tau,2h}(2\theta_t,2\theta_x) := 1-e^{-i2\theta_t} +\sigma(1-\cos2\theta_x)\;.
\end{equation*}

\paragraph{Action of $S_{2\tau,2h}$.}
Adapting the definition of the damped block Jacobi smoother~\eqref{eq:def-smoother} to the grid $(2\tau,2h)$, we get
\begin{equation*}
    S_{2\tau,2h} = I_{2\tau,2h} - \omega D_{2\tau,2h}^{-1}L_{2\tau,2h}\;,
\end{equation*}
where $D_{2\tau,2h}:=\text{diag}\{Q_{2\tau,2h}\}$. We may now use this expression to study the action of this smoothing operator on the Fourier space, which leads to
\begin{equation*}
    \begin{aligned}
        \widehat{S_{2\tau,2h}} &: \mathcal{F}_{2\tau,2h}(2\theta_t,2\theta_x) \to \mathcal{F}_{2\tau,2h}(2\theta_t,2\theta_x)\;, \\
        \widehat{S_{2\tau,2h}} &:=
        \begin{pmatrix}
            \hat{S}_{2\tau,2h}(2\check{\theta}_t,2\theta_x) & 0 \\ 
            0 & \hat{S}_{2\tau,2h}(2\hat{\theta}_{t},2\theta_{x})
        \end{pmatrix}
        ,
    \end{aligned}    
\end{equation*}
where we have used the notation
\begin{equation*}
    \hat{S}_{2\tau,2h}(2\theta_t,2\theta_x) := 1-\omega +\frac{\omega}{1+\sigma(1-\cos2\theta_x)}e^{-i2\theta_t}\;.
\end{equation*}

\paragraph{Action of $R_{2,1}$.}
In the context of the grid $(2\tau,2h)$, we have $R_{2,1}=R_{2\tau\to 4\tau}\otimes I_{2h}$. Therefore we obtain the following matrix operator acting on the Fourier space\MG{:}
\begin{equation*}
    \begin{aligned}
        \widehat{R_{2,1}} &: \mathcal{F}_{2\tau,2h}(2\theta_t,2\theta_x) \to \mathcal{F}_{4\tau,2h}(4\theta_t,2\theta_x)\;, \\
        \widehat{R_{2,1}} &:=
        \begin{pmatrix}
            \hat{R}(2\check{\theta}_t) & \hat{R}(2\hat{\theta}_{t})
        \end{pmatrix}
        .
    \end{aligned}
\end{equation*}

\paragraph{Action of $P_{2,1}$.}
As in the case of $P_{2,2}$, we can deduce the representation of $P_{2,1}$ in Fourier space from the one of $R_{2,1}$
\begin{equation*}
    \begin{aligned}
        \widehat{P_{2,1}} &: \mathcal{F}_{4\tau,2h}(4\theta_t,2\theta_x) \to \mathcal{F}_{2\tau,2h}(2\theta_t,2\theta_x)\;, \\
        \widehat{P_{2,1}} &:= 2 (\widehat{R_{2,1}})^T\;.
    \end{aligned}
\end{equation*}

\paragraph{Action of $L_{4\tau,2h}$.}
Considering again the rediscretization approach, we proceed as in the case of $L_{2\tau, 2h}$, which yields here the operator
\begin{equation*}
    \begin{aligned}
        \widehat{L_{4\tau,2h}} &: \mathcal{F}_{4\tau,2h}(4\theta_t,2\theta_x) \to \mathcal{F}_{4\tau,2h}(4\theta_t,2\theta_x)\;, \\
        \widehat{L_{4\tau,2h}} &:= \left( \hat{L}_{4\tau,2h}(4\theta_t,2\theta_x) \right) ,
    \end{aligned}    
\end{equation*}
where we have used the notation
\begin{equation*}
    \hat{L}_{4\tau,2h}(4\theta_t,2\theta_x) := 1-e^{-i4\theta_t} +2\sigma(1-\cos2\theta_x)\;.
\end{equation*}

Now we have studied the action of each operator involved in the expression of $M^o$, it is clear that the operator $\widehat{M}^o$ representing the action of $M^o$ on the space of harmonics verifies the following mapping property
\begin{equation*}
    \widehat{M}^o : \mathcal{E}_{\tau,h}(\theta_t,\theta_x) \to \mathcal{E}_{\tau,h}(\theta_t,\theta_x)\;.
\end{equation*}
Moreover, for each low frequency $(\theta_t,\theta_x)$, we are now in a position to compute the spectral radius $\rho^o(\theta_t,\theta_x)$ of this $8\times 8$ complex-valued matrix. Especially, we are interested in the convergence factor of the multigrid method obtained with this coarsening strategy, i.e. we would like to compute
\begin{equation}
\label{eq:sp-rad-Neumuller}
    \bar{\rho}^o := \max \left\{ \rho^o(\theta_t,\theta_x) \: | \: (\theta_t,\theta_x)\in\Theta^{low} \right\}\;.
\end{equation}

\subsection{New strategy}

For the simplest two-level V-cycle illustrated in Figure~\ref{fig:coarsening-strategies} (right), with $\nu_1$ and $\nu_2$ pre/post-smoothing steps, we get the multigrid iteration operator
\begin{equation}
    \label{eq:Mn}
    M^n = S_{\tau,h}^{\nu_{2}}\big[I_{\tau,h} - P_{4,2}(L_{4\tau,2h}^{-1})R_{4,2}L_{\tau,h}\big]S_{\tau,h}^{\nu_{1}} \;,
\end{equation}
where $P_{4,2}:=P_{2,2}P_{2,1}$ and $R_{4,2}:=R_{2,1}R_{2,2}=\frac{1}{4}P_{4,2}^T$.

Similarly to the analysis in Subsection~\ref{sec:strategie-neumuller}, we need to compute the Fourier symbols for each of the operators in~\eqref{eq:Mn}. For the operators $S_{\tau, h}$, $L_{\tau, h}$ and $L_{2\tau, 2h}$ please refer to the discussion in Subsection~\ref{sec:strategie-neumuller}. 

\paragraph{Action of $R_{4, 2}$.}
Given the definition of $R_{4,2}$ above, we have in the context of a fine grid $(\tau,h)$ that $R_{4,2} = (R_{2\tau\to 4\tau}R_{\tau\to 2\tau})\otimes R_{\tau\to 2\tau}$. By writing the explicit form of the operator $R_{2\tau\to 4\tau}R_{\tau\to 2\tau}$, we obtain 
\begin{equation*}
    \begin{aligned}
    R_{4, 2}\boldsymbol{\Phi}(\theta_t, \theta_x) = \hat{R}(\theta_t)\hat{R}(2\theta_t)\hat{R}(\theta_x)\boldsymbol{\Phi}(4\theta_t, 2\theta_x)\;,
    \end{aligned}
\end{equation*}
where we have used the trigonometric relation (obtained through Werner's formula)
\begin{equation*}
    \hat{R}(\theta_t)\hat{R}(2\theta_t)=\frac{1}{8}[\cos3\theta_t + 2\cos2\theta_t + 3\cos\theta_t + 2]\;.
\end{equation*}
We can therefore represent the action of this restriction operator on the space of harmonics by
\begin{equation*}
    \begin{aligned}
        \widehat{R_{4,2}} &: \mathcal{E}_{\tau,h}(\theta_t,\theta_x) \to \mathcal{F}_{4\tau,2h}(4\theta_t,2\theta_x)\;, \\
        \widehat{R_{4,2}} &:=
        \begin{pmatrix}
           \hat{R}(\check{\check{\theta}}_t)\hat{R}(2\check{\theta}_t)\hat{R}(\check{\theta}_x) & \: \cdots \: &  \hat{R}(\hat{\hat{\theta}}_{t})\hat{R}(2\hat{\theta}_{t})\hat{R}(\hat{\theta}_{x})
        \end{pmatrix}
        .
    \end{aligned}
\end{equation*}

\paragraph{Action of $P_{4, 2}$.}
By the definition of $P_{4, 2}$, we may represent it by
\begin{equation*}
    \begin{aligned}
        \widehat{P_{4, 2}}&:  \mathcal{F}_{4\tau, 2h}(4\theta_t, 2\theta_x) \to \mathcal{E}_{\tau, h}(\theta_t, \theta_x)\;, \\ 
        \widehat{P_{4, 2}} &:= 4\big(\widehat{R_{4, 2}}\big)^{\top}\;.
    \end{aligned}
\end{equation*}

\begin{remark}
    If $\eta_1 = \eta_2 = 0$ in the original coarsening strategy, then both coarsening strategies are equivalent. Indeed, in this case the approximation \eqref{eq:def-app-Lcoarseinv} of the inverse of $L_{2\tau,2h}$ becomes
    \begin{equation*}
        \widetilde{L_{2\tau,2h}^{-1}} = P_{2,1}(L_{4\tau,2h}^{-1})R_{2,1}\;.  
    \end{equation*}
    Then, since $P_{4,2}=P_{2,2}P_{2,1}$ and $R_{4,2}=R_{2,1}R_{2,2}$, it follows that $M^o=M^n$.
\end{remark}

As in the case of the previous strategy, it follows from these results that the operator $\widehat{M}^n$ representing the action of $M^n$ on the space of harmonics verifies the same mapping property
\begin{equation*}
    \widehat{M}^n : \mathcal{E}_{\tau,h}(\theta_t,\theta_x) \to \mathcal{E}_{\tau,h}(\theta_t,\theta_x)\;.
\end{equation*}
We may now study the following quantity, which depends on $\sigma$,
\begin{equation}
\label{eq:sp-rad-Neumuller}
    \bar{\rho}^n := \max \left\{ \rho^n(\theta_t,\theta_x) \: | \: (\theta_t,\theta_x)\in\Theta^{low} \right\}\;,
\end{equation}
where for each frequency $(\theta_t,\theta_x)$, $\rho^n(\theta_t,\theta_x)$ denotes the spectral radius of $\widehat{M}^n$.

For $\nu_1=\eta_1=3$ pre-smoothing and $\nu_2=\eta_2=3$ post-smoothing steps, we plotted in Figures~\ref{fig:rho(sigma)-orig-vs-new(default)} and \ref{fig:rho(sigma)-orig-vs-new} the dependence of $\bar{\rho}^o$ and $\bar{\rho}^n$ on $\sigma$. 
\begin{figure}
    \centering
    \includegraphics[width=\columnwidth]{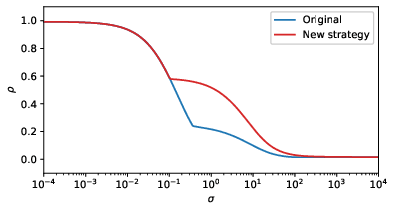}
    \caption{$\bar{\rho}^o$ and $\bar{\rho}^n$ as functions of $\sigma$, taking $\omega=\frac12$.}
    \label{fig:rho(sigma)-orig-vs-new(default)}
\end{figure}
\begin{figure}
    \centering
    \includegraphics[width=\columnwidth]{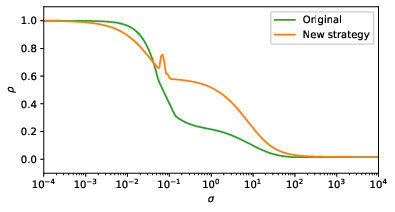}
    \caption{$\bar{\rho}^o$ and $\bar{\rho}^n$ as functions of $\sigma$, taking $\omega=\omega^*$ as in Theorem \ref{thm:optimal-w}.}
    \label{fig:rho(sigma)-orig-vs-new}
\end{figure}
In Figure~\ref{fig:rho(sigma)-orig-vs-new(default)}, we see that for $\omega=1/2$, $\bar{\rho}^o$ is always smaller than $\bar{\rho}^n$, which makes sense since there is an additional intermediate smoothing step in the original strategy (at the level $(2\tau,2h)$) compared to the new strategy.
Moreover, we also see that the convergence factors are very similar except for values of $\sigma$ approximately in $(0.02,400)$. This corresponds to the range of values of $\sigma$ for which the intermediate smoothing step has the highest impact.
Now in Figure~\ref{fig:rho(sigma)-orig-vs-new}, we plotted the two spectral radii again but this time choosing the damping parameter $\omega=\omega^*$ as prescribed by Theorem~\ref{thm:optimal-w}. For intermediate and high values of $\sigma$, the remarks made in the case $\omega=\frac12$ remain valid. But for low values of $\sigma$, more specifically for $\sigma\leq 0.04$, we see that $\bar{\rho}^n$ is significantly smaller than $\bar{\rho}^o$. This means that for such values of $\sigma$, the new strategy is not only cheaper than the original strategy but also faster. One possible explanation for this surprisingly nice property is that, when $\sigma$ is small, $\omega^*$ is close to the optimal damping parameter for the multi-level method. The reader is referred to Subsection \ref{subsec:opt-omega} for a more detailed discussion of this issue.

\begin{figure}
    \centering
    \includegraphics[width=0.8\columnwidth]{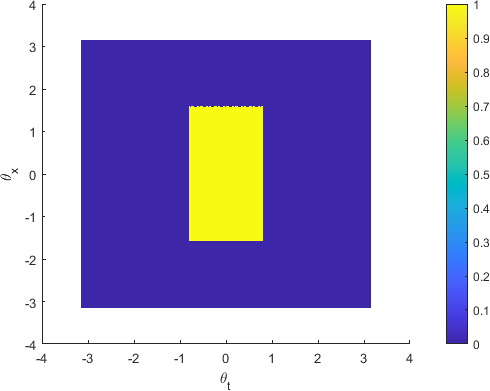}
    \caption{Representation of $\boldsymbol{V}_0$ as defined in~\eqref{eq:low-modes-input}}
    \label{fig:low-modes-input}
\end{figure}
\begin{figure}
    \centering
    \includegraphics[width=0.8\columnwidth]{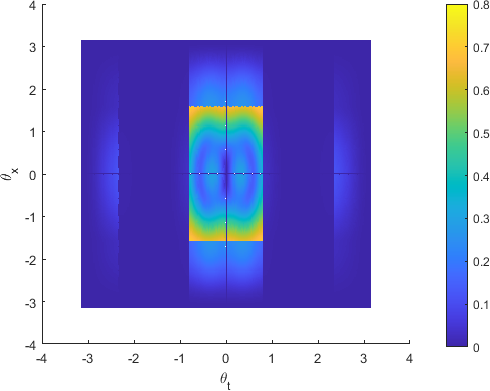}
    \includegraphics[width=0.8\columnwidth]{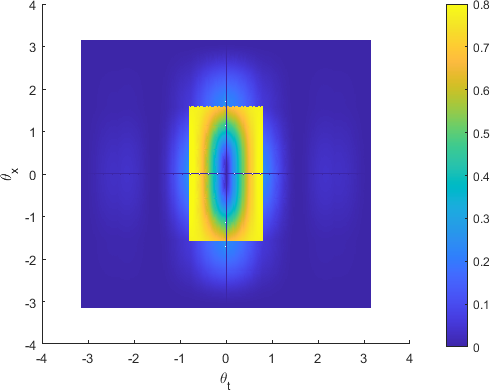}
    \caption{$\widehat{M}^o\boldsymbol{V}_0$ (top) and $\widehat{M}^n\boldsymbol{V}_0$ (bottom) for $\sigma=10^{-2}$}
    \label{fig:action-M-sigma-1e-2}
\end{figure}
\begin{figure}
    \centering
    \includegraphics[width=0.8\columnwidth]{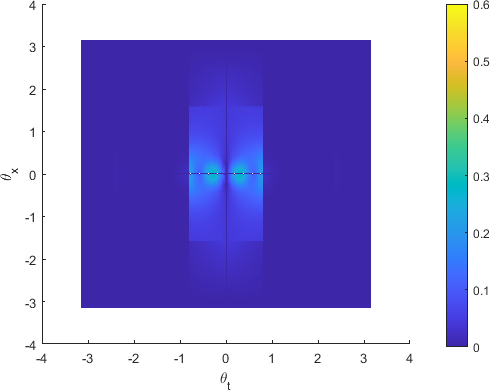}
    \includegraphics[width=0.8\columnwidth]{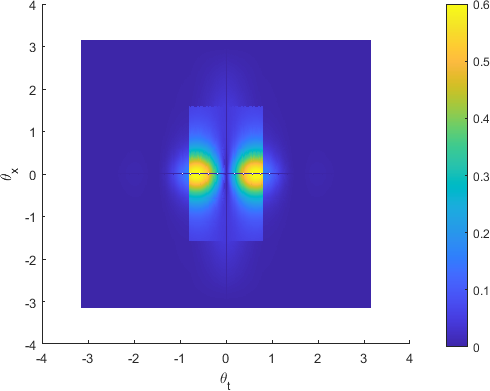}
    \caption{$\widehat{M}^o\boldsymbol{V}_0$ (top) and $\widehat{M}^n\boldsymbol{V}_0$ (bottom) for $\sigma=1$}
    \label{fig:action-M-sigma-1}
\end{figure}

\subsection{Action of the multigrid operators on $\Theta^{low}$}

In order to illustrate the action of the two-grid operators $M^o$ and $M^n$ on low frequency modes, we begin with choosing a simple linear combination of all low frequency modes as input
\begin{equation}
    \label{eq:low-modes-input}
    \boldsymbol{V}_0 := \sum_{(\theta_t,\theta_x)\in\Theta^{low}} \boldsymbol{\Phi}(\theta_t,\theta_x)\;.
\end{equation}
For any combination of the space-time frequency modes
\begin{equation*}
    \boldsymbol{V} := \sum_{(\theta_t,\theta_x)\in\Theta} v(\theta_t,\theta_x)\boldsymbol{\Phi}(\theta_t,\theta_x)\;,
\end{equation*}
with coefficients $v(\theta_t,\theta_x)\in\mathbb{C}$, we choose to represent it graphically by the moduli of its coefficients $|v|$ defined on the space-time frequency space $\Theta$.
Following this choice, we plotted the graph of $\boldsymbol{V}_0$ in Figure~\ref{fig:low-modes-input}, illustrating that $|v_0|$ takes the value~1 if $(\theta_t,\theta_x)\in\Theta^{low}$ and 0 otherwise. For $\nu_1=\eta_1=3$ pre-smoothing and $\nu_2=\eta_2=3$ post-smoothing steps, we have computed $\widehat{M}^o\boldsymbol{V}_0$ and $\widehat{M}^n\boldsymbol{V}_0$ for $\omega=\omega^*$ and different values of $\sigma$, and plotted the results in Figures~\ref{fig:action-M-sigma-1e-2} and \ref{fig:action-M-sigma-1}. According to these results, it appears that the intermediate smoothing step makes the original strategy more efficient than the new strategy when applied to the low frequency input $\boldsymbol{V}_0$ for both values $\sigma=1$ and $\sigma=10^{-2}$. Especially, the original strategy is much more efficient to damp space-time frequency modes $\boldsymbol{\Phi}(\theta_t,\theta_x)$ for which $\theta_t\in [-\frac{3\pi}{4},-\frac{\pi}{4}]\cup[\frac{\pi}{4},\frac{3\pi}{4}]$. An explanation for this is that the intermediate smoothing step is related to the coarsening from the finer grid $(2\tau,2h)$ to the coarser grid $(4\tau,2h)$. Therefore it is optimized to damp efficiently associated high frequency modes, i.e. modes for which $\theta_t\in[-\frac{\pi}{2},-\frac{\pi}{4}]\cup[\frac{\pi}{4},\frac{\pi}{2}]$. This damping effect also impacts the corresponding companion modes with respect to the coarsening $(\tau,h)\to(2\tau,2h)$, i.e. modes for which $\theta_t\in[-\frac{3\pi}{4},-\frac{\pi}{2}]\cup[\frac{\pi}{2},\frac{3\pi}{4}]$.

For $\sigma=10^{-2}$, we see in Figure~\ref{fig:action-M-sigma-1e-2} that the coefficients of $\widehat{M}^n\boldsymbol{V}_0$ are maximal near the boundaries of $\Theta^{low}$. Thus for the original strategy, the coefficients of $\widehat{M}^o\boldsymbol{V}_0$ near the boundaries $\left\{\theta_t=\pm \frac{\pi}{4}\right\}$ are well damped thanks to the effect of the intermediate smoothing step but the ones near the boundaries $\left\{\theta_x=\pm \frac{\pi}{2}\right\}$ are similar to the ones for the new strategy. Hence the global performances of both strategies are rather similar when applied to the input $\phi_0$.

Besides, in the case where $\sigma=1$, Figure~\ref{fig:action-M-sigma-1} shows that the coefficients of $\widehat{M}^n\boldsymbol{V}_0$ are maximal for the modes ${(\theta_t,\theta_x)=\left(\pm\frac{\pi}{4},0 \right)}$. So the damping effect of the intermediate smoothing step is more important than in the case $\sigma=10^{-2}$, and we end up with maximal coefficients of $\widehat{M}^o\boldsymbol{V}_0$ about half of the ones of $\widehat{M}^n\boldsymbol{V}_0$. Thus the original strategy is much more efficient than the new one in this case.

\subsection{Optimizing the damping parameter for the multi-level method}
\label{subsec:opt-omega}

In the previous subsections (except from Figure~\ref{fig:rho(sigma)-orig-vs-new(default)}), we had chosen for each strategy a damping parameter equal to $\omega^*$ according to Theorem~\ref{thm:optimal-w} in order to maximize the efficiency of the smoothing step. However, this choice might not be optimal when considering the multi-level strategies presented above. Indeed, there may exist values of $\omega$ such that the spectral radii $\bar{\rho}^o$ and $\bar{\rho}^n$ are minimal, for given values $\nu_1$, $\nu_2$, $\eta_1$ and $\eta_2$. Let us denote these values of damping parameter by $\omega_{opt}^o$ and $\omega_{opt}^n$. By definition, we have 
\begin{equation*}
        \omega_{opt}^o := \argmin_{\omega \in [0,1]} \bar{\rho}^o\;, \quad
        \omega_{opt}^n := \argmin_{\omega \in [0,1]} \bar{\rho}^n\;.
\end{equation*}
Since the expressions of $\bar{\rho}^o$ and $\bar{\rho}^o$ are already quite intricate, we do not seek to find the analytical expressions of $\omega_{opt}^o$ and $\omega_{opt}^n$. Instead, we compute them numerically as a function of $\sigma$, for fixed numbers of smoothing steps $\nu_1$, $\nu_2$, $\eta_1$ and $\eta_2$.

\begin{figure}
    \centering
    \includegraphics[width=\columnwidth]{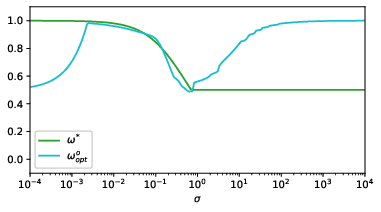}
    \caption{$\omega_{opt}^o$ and $\omega^*$ as functions of $\sigma$}
    \label{fig:omega-opt-vs-star(orig)}
\end{figure}

\begin{figure}
    \centering
    \includegraphics[width=\columnwidth]{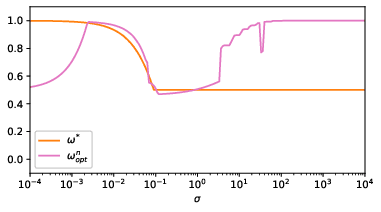}
    \caption{$\omega_{opt}^n$ and $\omega^*$ as functions of $\sigma$}
    \label{fig:omega-opt-vs-star(new)}
\end{figure}

\begin{figure}
    \centering
    \includegraphics[width=\columnwidth]{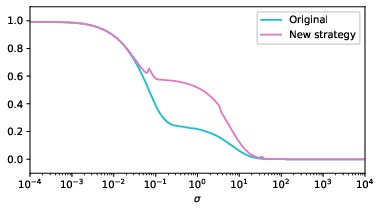}
    \caption{$\bar{\rho}^o$ and $\bar{\rho}^n$ as functions of $\sigma$, taking $\omega=\omega_{opt}$ as in Figures \ref{fig:omega-opt-vs-star(orig)} and \ref{fig:omega-opt-vs-star(new)}.}
    \label{fig:rho(sigma)-opt-orig-vs-new}
\end{figure}

In Figures \ref{fig:omega-opt-vs-star(orig)} and \ref{fig:omega-opt-vs-star(new)}, we plotted the values of $\omega_{opt}^o$ and $\omega_{opt}^n$ with respect to $\sigma$ for $\nu_1=\nu_2=\eta_1=\eta_2=3$, and we compared them to $\omega^*$ for each strategy. Then, in Figure \ref{fig:rho(sigma)-opt-orig-vs-new}, we have compared $\bar{\rho}^o$ and $\bar{\rho}^n$ when taking the optimal damping parameters $\omega_{opt}^o$ and $\omega_{opt}^n$ computed numerically. We see that for these optimal versions, the original strategy always performs better than the new one, which was expected since it contains additional intermediate pre/post-smoothing steps. Thus when taking the best out of each strategy, we obtain a more efficient (though more expensive) algorithm with the original strategy.

\section{Numerical experiments}
\label{sec:num-exp}

To illustrate our theory, we compute the solution to the heat equation~\eqref{eq:continuous-heat-equation} on the time interval $[0, 0.1]$ with the source term 
\begin{equation*}
    f(x, t) = x^4(1-x)^4 + 10\sin(8t)\;,
\end{equation*}
and initial value ${u_{0}(x) = 0}$. We start the iterative process with a random initial guess. 

We solve the problem in two different regimes: first for $\sigma$ small (Figure~\ref{fig:stmg-conv-small-sigma}, where both coarsening strategies with optimal parameters have an efficiency close to $2$), then for $\sigma$ large (Figure~\ref{fig:stmg-conv-big-sigma}, where both the default and optimal strategies share the same smoothing factor $\omega^* = 0.5$).
In all cases, $\nu_{1} = \eta_{1} = 3$ pre-smoothing and $\nu_{2} = \eta_{2} = 3$ post-smoothing steps were performed. 
The error is measured in the norm $L^\infty(0,T;L^2(\Omega))$.

For $\sigma\approx 1.56\times 10^{-1}$, we observe in Figure~\ref{fig:stmg-conv-small-sigma-unnormalized} that for the default parameter $\omega=0.5$, the convergence is quite slow for both strategies. We also see that the errors for the two strategies are very close (the red line is close to the blue line). 
Besides, for this value of $\sigma$ and the original coarsening strategy, choosing the damping parameter in an optimal fashion (see Theorem~\ref{thm:optimal-w}) yields some gains in terms of efficiency.
With this "optimal" version of the algorithm, we also observe that after 5 iterations of the new strategy, we get an error similar to the one when $\omega=0.5$ is used for 10 iterations. In terms of computing time, this is a gain of approximately $30\%$.
However, it is worth noting that even further gains can be achieved by computing an optimal two-level damping parameter. 
In the considered regime, the new coarsening strategy is supposed to be much worse than the original one (see Figures~\ref{fig:rho(sigma)-orig-vs-new(default)} and~\ref{fig:rho(sigma)-orig-vs-new}), which is what we observe for all damping parameter choices. So using the new coarsening strategy is more beneficial when the value of $\sigma$ is larger, which is what we encounter in practical situations.

\begin{figure}
    \centering
    \begin{subfigure}[t]{\columnwidth}
        \centering
        \includegraphics[width=\textwidth]{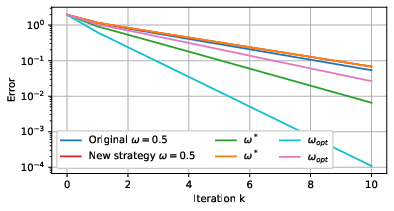}
        \caption{Error for ten iterations}
        \label{fig:stmg-conv-small-sigma-unnormalized}
    \end{subfigure}
    \begin{subfigure}[t]{\columnwidth}
        \centering
        \includegraphics[width=\textwidth]{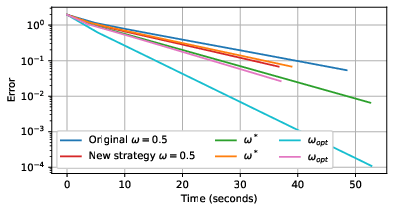}
        \caption{Error as a function of run-time (precomputations excluded)}
        \label{fig:stmg-conv-small-sigma-normalized}
    \end{subfigure}
    \caption{Convergence of the algorithm for a small ratio $\sigma \approx 1.56\times 10^{-1}$}
    \label{fig:stmg-conv-small-sigma}
\end{figure}

From Figure~\ref{fig:stmg-conv-big-sigma-unnormalized}, we notice that in the regime where $\sigma$ is large, the convergence rate of the new coarsening strategy is slightly lower than for the original strategy, which is to be expected given the values of the convergence factors in Figures~\ref{fig:rho(sigma)-orig-vs-new(default)} and~\ref{fig:rho(sigma)-orig-vs-new}. 
However, when plotting the error as a function of the computational time (see Figure~\ref{fig:stmg-conv-big-sigma-normalized}), we observe a quicker convergence for the new coarsening strategy. This is a very interesting optimization as the original STMG method presented in \cite{gander2016analysis} to which we compare our new strategy is considered to be extremely efficient in this regime. The results have been obtained using a sequential code and, in practice, the smoothing procedure would be computed in parallel, making it cheaper. Numerically, we observe a benefit to using the new strategy as soon as it is approximately $15\%$ cheaper than the original method in computational cost per iteration. 

\begin{figure}
    \centering
        \begin{subfigure}[t]{\columnwidth}
            \centering
            \includegraphics[width=\textwidth]{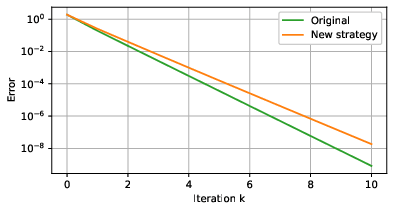}
            \caption{Error for ten iterations}
            \label{fig:stmg-conv-big-sigma-unnormalized}
        \end{subfigure}
        \begin{subfigure}[t]{\columnwidth}
            \centering
            \includegraphics[width=\textwidth]{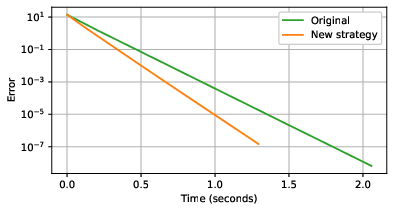}
            \caption{Error as a function of run-time (precomputations excluded)}
            \label{fig:stmg-conv-big-sigma-normalized}
        \end{subfigure}
    \caption{Convergence of the algorithm for a big ratio $\sigma \approx 645$}
    \label{fig:stmg-conv-big-sigma}
\end{figure}



\section{Conclusion}

We derived optimal smoothing factors for the Space-Time Multigrid
algorithm proposed in~\cite{gander2016analysis}, demonstrated
their efficiency compared to the original method and verified it
numerically.  Additionally, the original approach suggests to do a
coarsening by a factor $2$ in space and in time as long as
$\tau/h^2>1/\sqrt{2}$ holds, and then to proceed to do coarsening in
time only. In order to prevent this change of regimes, we suggested
to coarsen by a factor $2$ in space and $4$ in time. Our theoretical
analysis revealed that this new strategy leads to a multi-level
algorithm which has a comparable convergence factor to the original strategy, except for intermediate values of the ratio $\sigma$, where it is slightly bigger.
However, this new strategy is always competitive in practice as it is cheaper than the original one. Indeed, since it does not require any intermediate smoothing steps, it is slightly slower for intermediate values of $\sigma$, and it is faster for larger values of $\sigma$, where the original algorithm is known to perform extremely well. 
Finally, we also designed a numerical technique to optimize the multi-level method with respect to the damping parameter, therefore leading to an optimal version of the STMG algorithm in this context. 

As far as future work is concerned, it would be interesting to extend
the multigrid analysis to the case of higher order Runge Kutta methods
for the time discretization. Another natural extension would be to the heat equation in
higher dimension. Since the size of the problem
is then bigger, we expect that the acceleration obtained with
our new strategy would be even more beneficial.

\section*{Acknowledgements}
This project has received funding from the Federal Ministry of Education and Research and the European High-Performance Computing Joint Undertaking (JU) under grant agreement No 955701, Time-X. The JU receives support from the European Union’s Horizon 2020 research and innovation program and Belgium, France, Germany, Switzerland.

\bibliographystyle{spmpsci}      
\bibliography{refs}   

\end{document}